\documentclass[times]{nlaauth_preprint}%
\usepackage{moreverb}
\usepackage{amsfonts}
\usepackage{amsmath}
\usepackage{amssymb}
\usepackage{graphicx}
\usepackage{overpic}
\usepackage{multirow}
\usepackage{subfigure}
\usepackage{url}
\usepackage{nicefrac}
\usepackage[pdftex]{thumbpdf}
\usepackage[pdftex,
pdfstartview=FitH,bookmarks=true,bookmarksnumbered=true,bookmarksopen=true,hypertexnames=false,breaklinks=true,
colorlinks=true,  linkcolor=blue,anchorcolor=blue,
citecolor=blue,filecolor=blue,  menucolor=blue]{hyperref}%
\providecommand{\U}[1]{\protect\rule{.1in}{.1in}}
\providecommand{\U}[1]{\protect\rule{.1in}{.1in}}
\newtheorem{theorem}{Theorem}
\newtheorem{algorithm}[theorem]{Algorithm}
\newtheorem{assumption}[theorem]{Assumption}

\newtheorem{lemma}[theorem]{Lemma}

\newtheorem{remark}[theorem]{Remark}

 \definecolor{chl}{RGB}{0,0,205}  % darker blue
 \definecolor{chl2}{RGB}{205,0,0}  % darker red

\def\ol{\overline}

\def\vc#1{\mathbf{#1}}

\begin{document}

\title{BDDC for Mixed-Hybrid Formulation of Flow in Porous Media with Combined Mesh Dimensions}
\author{
Jakub~{\v S}{\'\i}stek\affil{1}, 
Jan~B{\v r}ezina\affil{2} and
Bed\v{r}ich~Soused\'{i}k\affil{3}
}
\runningheads{J. {\v S}{\'{i}}stek, J.~B{\v r}ezina, B.~Soused\'{i}k}
{BDDC for Porous Media with Combined Mesh Dimensions}
\address{
\affilnum{1}
Institute of Mathematics,
Academy of Sciences of the Czech Republic,
\v Zitn\' a 25, 115~67~Prague~1, Czech Republic
\break\affilnum{2}
Institute for Nanomaterials, Advanced Technology and Innovation,
Technical University of Liberec,
Bendlova~1407/7,~461~17~Liberec~1, Czech Republic
\break\affilnum{3}
Department of Mathematics and Statistics, 
University of Maryland,
Baltimore County, 
1000~Hilltop~Circle,~Baltimore,~MD~21250,~USA
}
\cgs{
The research was supported 
by Czech Science Foundation under project \mbox{GA \v{C}R 14-02067S},
by Academy of Sciences of the Czech Republic through \mbox{RVO:67985840},
and 
by Ministry of Education, Youth and Sports under project CZ.1.05/2.1.00/01.0005.
J.~{\v S}\'{\i}stek acknowledges the computing time 
on \emph{HECToR} provided through the PRACE-2IP project (FP7 RI-283493). 
}
\corraddr{J.~{\v S}{\'\i}stek,
Institute of Mathematics,
Academy of Sciences of the Czech Republic,
\v Zitn\' a~25, 115~67~Prague~1, Czech Republic.
E-mail: sistek@math.cas.cz}
\begin{abstract}
We extend the Balancing Domain Decomposition by Constraints (BDDC) method
to flows in porous media discretised by mixed-hybrid finite elements with combined mesh dimensions.
Such discretisations appear when major geological fractures are modelled
by 1D or 2D elements inside three-dimensional domains.
In this set-up, the global problem as well as the substructure problems have a symmetric saddle-point structure,
containing a `penalty' block due to the combination of meshes.
We show that the problem can be reduced by means of iterative substructuring
to an interface problem, which is symmetric and positive definite.
The interface problem can thus be solved by conjugate gradients with the BDDC method as a preconditioner.
A~parallel implementation of this algorithm is incorporated into an existing software package for subsurface flow simulations.
We study the performance of the iterative solver on several academic and real-world problems.
Numerical experiments illustrate its efficiency and scalability.
\end{abstract}

\maketitle

\section{Introduction}

\label{sec:introduction} A detailed description of flow in porous media is
essential for building mathematical models with applications in, for example, water
management, oil and gas recovery, carbon dioxide (CO$_{2}$) sequestration or
nuclear waste disposal. In order to set up a reliable numerical model, one needs to
have a good knowledge of the problem geometry and input parameters.
For example, the flow of water in granite rock, which is a suitable site for 
nuclear waste disposal, is driven by the complex system
of vugs, cavities and fractures with various topology and sizes. These alter
the effective permeability, and therefore should be accurately accounted for
in the numerical model. There are two main approaches: either the fractures are
considered as free-flow regions, or the fractures contain debris and are
also modelled as porous media with specific permeabilities. In the first case,
a~unified approach to modelling free-flow and porous media regions can be
provided by the so called Stokes-Brinkman equation, which reduces to either
the Stokes or Darcy model in certain parameter limits, e.g., within the Multiscale
Mixed Finite-Element (MsMFE) framework~\cite{Gulbransen-2010-MMF}. In this
paper, we consider the latter case, and apply Darcy's law to the flow in
the reservoir and in the fractures as well; see~\cite{Martin-2005-MFB} for a
related approach. In either case, the preferential flow in large geological
dislocations and their intersections should be considered as two- and
one-dimensional flows, respectively. 
Due to the quite complex structure of the domains, 
the discretisation is performed using finite element methods~(FEM).
The resulting meshes are therefore unstructured, and they combine different
spatial dimensions (line elements in 1D, triangles in 2D, and tetrahedrons in
3D). The systems of linear equations obtained from the FEM discretisation are
often very large, so that using direct methods is prohibitive and iterative
solvers are warranted. The systems are typically also
{bad}-conditioned due to the mixing of spatial dimensions, large jumps in
permeability coefficients and presence of elements of considerably different
sizes, and so they are challenging for iterative solvers as well.

The matrices have a saddle-point structure%

\begin{equation}
\left[
\begin{array}
[c]{cc}%
A & \ol{B}^{T}\\
\ol{B} & -\ol{C}
\end{array}
\right]  , \label{eq:sp-matrix}%
\end{equation}
{where $A$ is symmetric positive definite on the kernel of $\ol{B}$,
$\ol{C}$ is symmetric positive semi-definite, and it is positive definite on the kernel of $\ol{B}^T$.}
The `penalty' 
block $\ol{C}\neq0$ arises from connecting meshes of
different spatial dimensions. The iterative solution of systems with this
structure is a frequently studied topic; see,
for example,~\cite{Benzi-2005-NSS,Dohrmann-2006-PBP,Tu-2005-BAM,Tu-2007-BAF}, the
monographs~\cite{Vassilevski-2008-MBF,Elman-2005-FEF} or~\cite[Chapter~9]%
{Toselli-2005-DDM} and the references therein. 
However, efficient methodologies for solving saddle-point problems are typically problem dependent.

In this paper, we develop a robust and scalable solver for linear systems with
the saddle-point structure as in~(\ref{eq:sp-matrix}) with the block$~\ol{C}$
either zero or nonzero. The solver is tailored to the mixed-hybrid formulation of
flow in porous media using the lowest order Raviart-Thomas ($RT_{0}$) finite
elements with combined mesh dimensions (1D, 2D and 3D). In particular, we
adapt the Balancing Domain Decomposition by Constraints (BDDC) method to this
type of problems.

The BDDC method is currently one of the most popular methods of iterative
substructuring. It has been proposed independently in
\cite{Cros-2003-PSC,Dohrmann-2003-PSC,Fragakis-2003-MHP}; 
see~\cite{Mandel-2007-BFM,Sousedik-2008-EPD} for the proof of equivalence.
Even though BDDC has been originally formulated for elliptic problems, it
has been successfully extended, for example, beyond elliptic
cases~\cite{Li-2006-BAI,Tu-2008-BDD} and to multiple
levels~\cite{Tu-2007-TBT3D,Mandel-2008-MMB}. An optimal set-up has been
studied
in~\cite{Klawonn-2008-AFA,Mandel-2007-ASF,Sistek-2012-FSC,Sousedik-2013-AMB}.
A closely related BDDC preconditioner for vector field problems discretised with Raviart-Thomas finite elements has been studied in~\cite{Oh-2013-BAR}.

We are interested in applications of the BDDC method to
saddle-point problems. If $\ol{C}=0$ in~(\ref{eq:sp-matrix}), one possible approach
is to use an algebraic trick and constrain the iterative solution of the
indefinite problem into a \emph{balanced} subspace, which is sometimes also
called \emph{benign}, where the operator is positive definite; 
see~\cite{Li-2006-BAI} for the Stokes problem,
and~\cite{Tu-2005-BAM,Sousedik-2013-NBS,Tu-2011-TBA} for flow in porous media.
However, due to the mixed-hybrid formulation and possible coupling of meshes
with different spatial dimensions, $\ol{C}\neq0$ in general, and we will favour 
an alternative, \emph{dual} approach here.

Our methodology is as follows. The mixed-hybrid formulation \cite{Maryska-1995-MHF,Oden-1977-DMH} is used in order
to modify the saddle-point problem to one which is symmetric and positive
definite {by means of iterative substructuring. 
In particular, we introduce a symmetric positive definite Schur complement with respect to
interface Lagrange multipliers, corresponding to a part of block~$\ol{C}$.}
The
reduced system is solved by the preconditioned conjugate gradient (PCG)
method, and the BDDC method is used as a preconditioner. From this
perspective, our work can be viewed as a further extension
of~\cite{Tu-2007-BAF}. Our main effort here is in accommodating the BDDC
solver to flows in porous media with combined mesh dimensions. 
In addition, the presentation of the BDDC algorithm is driven more by an efficient implementation,
while it is more oriented towards underlying theory in \cite{Tu-2007-BAF}.
We take
advantage of the special structure of the blocks in matrix~(\ref{eq:sp-matrix}%
) studied in detail in~\cite{Maryska-1995-MHF,Maryska-2000-SCS,
Maryska-2005-NSF}. 
In particular,
the nonzero structure of block $\ol{C}$ resulting from a combination of meshes with
different spatial dimensions is considered
in \cite{Brezina-2010-MHF}. 
We describe our parallel
implementation of the method and study its performance on several benchmark
and real-world problems. Another original contribution of this paper is
proposing a new scaling operator in the BDDC method suitable for the studied
problems. We note that if there is no coupling of meshes with different
spatial dimensions present in the discretisation, the block $\ol{C}=0$%
\ in~(\ref{eq:sp-matrix}) and our method is almost identical to the one introduced
in~\cite{Tu-2007-BAF}.

The paper is organised as follows. In Section~\ref{sec:model}, we introduce
the model problem. In Section~\ref{sec:fractures}, we describe the modelling of
fractured porous media and combining meshes of different dimensions. In
Section~\ref{sec:substructuring}, we introduce the substructuring components
and derive the interface problem. In Section~\ref{sec:bddc}, we formulate the
BDDC preconditioner. 
In addition, the selection of interface weights for BDDC is studied in detail in Section~\ref{sec:scaling}.
In Section~\ref{sec:solver}, we describe our parallel
implementation, and in Section~\ref{sec:numerical} we report the numerical
results and parallel performance for benchmark and engineering problems.
Finally, Section~\ref{sec:conclusion} provides a summary of our work.

Our notation does not, for simplicity, distinguish between finite element
functions and corresponding algebraic vectors of degrees of freedom, and
between linear operators and matrices within a specific basis---the meaning should be
clear from the context. The transpose of a matrix is denoted by superscript $^{T}$ and the energy
norm of a vector$~x$ is denoted\ by $\left\Vert x\right\Vert _{M}=\sqrt
{x^{T}Mx}$, where~$M$ is a symmetric positive definite matrix.

\section{Model problem}

\label{sec:model}

Let $\Omega$ be an open bounded {polyhedral} domain in $\mathbb{R}^{3}$. We are
interested in the solution of the following problem, combining Darcy's law
and the equation of continuity written as
\begin{align}
\Bbbk^{-1}\mathbf{u}+\nabla p  &  =-\nabla z\quad\text{in }\Omega
,\label{eq:problem-1}\\
\nabla\cdot\mathbf{u}  &  =f\quad\text{in }\Omega,\label{eq:problem-2}\\
p  &  =p_{N}\quad\text{on }\partial\Omega_{N},\label{eq:problem-3}\\
\mathbf{u}\cdot\mathbf{n}  &  =0\quad\text{on }\partial\Omega_{E},
\label{eq:problem-4}%
\end{align}
subject to boundary conditions on $\partial\Omega=\overline{\partial\Omega
}_{N}\cup\overline{\partial\Omega}_{E}$, where $\partial\Omega_{N}$ stands for
the part of the boundary with \emph{natural} (Dirichlet) boundary condition,
and $\partial\Omega_{E}$ for the part with \emph{essential} (Neumann) boundary
condition. In applications, the variable$~\mathbf{u}$ describes the velocity
of the fluid and $p$ the pressure (head) in an aquifer $\Omega$, $\Bbbk$ is a
symmetric positive definite tensor of the hydraulic conductivity, $-\nabla
z=(0,0,-1)^{T}$ is the gravity term, and $\mathbf{n}$ is the outer unit normal
vector of $\partial\Omega$. The term $\nabla z$ is present due to the fact,
that $\mathbf{u}$ satisfies $\mathbf{u}=-\Bbbk\nabla p_{h}$, where $p_{h}=p+z$
is the piezometric head. For a thorough\ discussion of application background
we refer, e.g., to monographs~\cite{Bear-1988-DFP,Chen-2006-CMM}.

Let $\mathcal{T}$ be the triangulation of domain $\Omega$ consisting of
$N_{E}$ simplicial elements with characteristic size $h$. We introduce a
space
\begin{equation}
\mathbf{H}(\Omega;\operatorname{div})=\left\{  \mathbf{v:v}\in L^{2}%
(\Omega);\ \nabla\cdot\mathbf{v}\in L^{2}(\Omega)\text{ and }\mathbf{v}%
\cdot\mathbf{n}=0\text{ on }\partial\Omega_{E}\right\}  ,
\end{equation}
equipped with the standard norm. Let $\mathbf{V}\subset\mathbf{H}%
(\Omega,\operatorname{div})$ be the space consisting of the lowest order
Raviart-Thomas ($RT_{0}$) functions and let $Q\subset L^{2}(\Omega)$ be the
space consisting of piecewise constant functions on the elements of the
triangulation $\mathcal{T}$. We refer, e.g., to 
monograph~\cite{Brezzi-1991-MHF} for a detailed description of the mixed
finite elements and the corresponding spaces.

In the \emph{mixed finite element approximation} of problem
(\ref{eq:problem-1})--(\ref{eq:problem-4}) we look for a pair $\left\{
\mathbf{u},p\right\}  \in\mathbf{V}\times Q$ that satisfies
\begin{align}
\int_{\Omega}\Bbbk^{-1}\mathbf{u}\cdot\mathbf{v}\,dx-\int_{\Omega}p\nabla
\cdot\mathbf{v}\,dx  &  =-\int_{\partial\Omega_{N}}p_{N}\mathbf{v}%
\cdot\mathbf{n}\,ds-\int_{\Omega}v_{z}\,dx,\quad\forall\mathbf{v}\in
\mathbf{V},\label{eq:mixed-1}\\
-\int_{\Omega}q\nabla\cdot\mathbf{u}\,dx  &  =-\int_{\Omega}fq\,dx,\quad
\forall q\in Q. \label{eq:mixed-2}%
\end{align}

In the discrete formulation, we need $p_{N}$ and $f$ only sufficiently regular
so that the integrals in the weak formulation~(\ref{eq:mixed-1}%
)--(\ref{eq:mixed-2}) make sense, namely $p_{N}\in L^{2}\left(  \partial
\Omega_{N}\right)  $, $f\in L^{2}\left(  \Omega\right)  $.

Next, we recall the mixed-hybrid formulation. It was originally motivated by
an effort to modify the saddle-point problem (\ref{eq:mixed-1}%
)--(\ref{eq:mixed-2}) to one which leads to symmetric positive definite
matrices. Nevertheless, this formulation is also suitable for a combination of
meshes with different spatial dimensions, which will be described in detail in
the next section.

Let $\mathcal{F}$ denote the set of inter-element \emph{faces} of the
triangulation $\mathcal{T}$. We now introduce several additional spaces.
First, let us define the space $\mathbf{V}^{-1}$ by relaxing the condition of
continuity of the normal components in the space $\mathbf{V}$ on inter-element
boundaries $\mathcal{F}$. More precisely, we define local spaces
$\mathbf{V}^{i}$ for each element $T^{i}\in\mathcal{T}$, $i=1,\ldots,N_{E}$,
by
\begin{equation}
\mathbf{V}^{i}=\left\{  \mathbf{v}\in\mathbf{H}(T^{i};\operatorname{div}%
):\mathbf{v}\in RT_{0}(T^{i})\right\}  ,
\end{equation}
and put $\mathbf{V}^{-1}=\mathbf{V}^{1}\times\cdots\times\mathbf{V}^{N_{E}}$.
Next, we define the space of Lagrange multipliers $\Lambda$ consisting of
functions that take constant values on individual inter-element faces in
$\mathcal{F}$,
\begin{equation}
\Lambda=\left\{  \lambda\in L^{2}\left(  \mathcal{F}\right)  :\lambda
=\mathbf{v}\cdot\mathbf{n}|_{\mathcal{F}},\;\mathbf{v\in V} \right\}.
\label{eq:lambda_space}
\end{equation}
In particular, $\lambda=0$ on $\partial\Omega$ for any $\lambda\in\Lambda$.

In the \emph{mixed-hybrid finite element approximation} of problem
(\ref{eq:problem-1})--(\ref{eq:problem-4}), we look for a triple $\left\{
\mathbf{u},p,\lambda\right\}  \in\mathbf{V}^{-1}\times Q\times\Lambda$ that
satisfies
\begin{align}
\sum_{i=1}^{N_{E}}\left[  \int_{T^{i}}\Bbbk_{i}^{-1}\mathbf{u}\cdot
\mathbf{v}\,dx-\int_{T^{i}}p\nabla\cdot\mathbf{v}\,dx+\int_{\partial
T^{i}\setminus\partial\Omega}\lambda(\mathbf{v}\cdot\mathbf{n})|_{\partial
T_{i}}\,ds\right]   &  \label{eq:hybrid-1}\\
= -\int_{\partial\Omega_{N}}p_{N}\mathbf{v}\cdot\mathbf{n}\,ds-\sum_{i=1}%
^{N_{E}}\int_{T^{i}}v_{z}\,dx,\quad\forall\mathbf{v}  &  \in\mathbf{V}^{-1}%
,\nonumber\\
-\sum_{i=1}^{N_{E}}\left[  \int_{T^{i}}q\nabla\cdot\mathbf{u}\,dx\right]   &
=-\int_{\Omega}fq\,dx,\quad\forall q\in Q,\label{eq:hybrid-2}\\
\sum_{i=1}^{N_{E}}\left[  \int_{\partial T^{i}\setminus\partial\Omega}%
\mu(\mathbf{u}\cdot\mathbf{n})|_{\partial T_{i}}\,ds\right]   &
=0,\quad\forall\mu\in\Lambda. \label{eq:hybrid-3}%
\end{align}
Equation (\ref{eq:hybrid-3}) imposes a continuity condition on the normal
component of the velocity (also called \emph{normal flux}) $\mathbf{u}\cdot\mathbf{n}$ across $\mathcal{F}$ which guarantees that
$\mathbf{u}\in\mathbf{V}$. This condition also implies the equivalence of the
two formulations (\ref{eq:mixed-1})--(\ref{eq:mixed-2}) and (\ref{eq:hybrid-1}%
)--(\ref{eq:hybrid-3}). We note that the Lagrange multipliers $\lambda$ can be
interpreted as the approximation of the trace of $p$ on $\mathcal{F}$,
see~\cite{Cowsar-1995-BDD} for details.

Let us now write the matrix formulation corresponding to (\ref{eq:hybrid-1}%
)--(\ref{eq:hybrid-3}) as
\begin{equation}
\left[
\begin{array}
[c]{ccc}%
A & B^{T} & B_{\mathcal{F}}^{T}\\
B & 0 & 0\\
B_{\mathcal{F}} & 0 & 0
\end{array}
\right]  \left[
\begin{array}
[c]{c}%
\mathbf{u}\\
p\\
\lambda
\end{array}
\right]  =\left[
\begin{array}
[c]{c}%
g\\
f\\
0
\end{array}
\right]  . \label{eq:hybrid-m}%
\end{equation}
It is important to note that $A$ is block diagonal with $N_{E}$ blocks,
corresponding to elements $T^{i}$, $i=1,\dots,N_{E}$, and each of the blocks
is symmetric positive definite, cf. the first term in~(\ref{eq:hybrid-1}).
It was shown in~\cite{Maryska-2000-SCS} that the system of equations (\ref{eq:hybrid-m})
can be reduced (twice) to the Schur complement corresponding to the Lagrange
multipliers $\lambda$ and solved efficiently by a direct or iterative solver.
Here, we will look for an efficient solution of a slightly modified, and in
general also block dense, system which is introduced in the next section.

\section{Modelling of fractures}

\label{sec:fractures} In this section, we recall the main ideas of the
discrete model of the flow in fractured porous media that is based on
connection of meshes of different dimensions as described
in~\cite{Brezina-2010-MHF}. Let us denote the full domain by $\Omega
_{3}=\Omega$. Next, consider lower-dimensional domains $\Omega_{d-1}%
\subset\Omega_{d}$, $d=2,3$, such that $\Omega_{2}$ consists of polygons and
$\Omega_{1}$ consists of line segments. We will also assume that
$\partial\Omega_{1}\subset\partial\Omega_{2}\subset\partial\Omega_{3}$. 
The first condition requires that a domain of a lower dimension cannot poke out of
the domain of higher dimension, 
while the second condition prevents domains of lower dimension from having boundaries  
in the interior of domains of higher dimension.
We impose these conditions to avoid technical difficulties in the analysis. 
However, numerical
evidence suggests that these conditions are not necessary, and in fact, they are not
satisfied for the real-world problems presented in
Section~\ref{sec:results_engineering}.

For every dimension $d=1,2,3$, we introduce a triangulation $\mathcal{T}_{d}$
of the domain $\Omega_{d}$ that consists of finite elements $T_{d}^{i},$\ $i =
1,\dots,N_{E}^{d}$ and satisfies the compatibility conditions
\begin{gather}
T_{d-1}^{i} \subset\mathcal{F}_{d}, \quad\text{where } \mathcal{F}_{d}=
\bigcup_{k} \partial T_{d}^{k} \setminus\partial\Omega_{d},\\
\quad T_{d-1}^{i} \cap\partial T^{j}_{d} \text{ is either $T_{d-1}^{i}$ or
$\emptyset$},
\end{gather}
for every $i\in\{1,\dots, N_{E}^{d-1}\}$, $j\in\{1,\dots,N_{E}^{d}\}$, and
$d=2,3$.
This means that elements of a lower dimension match faces of elements of the higher dimension.

We consider equations (\ref{eq:problem-2})--(\ref{eq:problem-4}) on the domains
$\Omega_{d}$, $d=1,2,3$, completed by a slight modification of the Darcy's
law~(\ref{eq:problem-1}):
\begin{equation}
\label{eq:darcy-law-fracture}
\Bbbk_{d}^{-1}\frac{\mathbf{u}_{d}}{\delta_{d}}+\nabla p_{d}=-\nabla z,
\end{equation}
where $\mathbf{u}_{d}$ stands for the 
velocity integrated over the cross-section for $d=1,2$,
i.e. the units of $\mathbf{u}_{3}$, $\mathbf{u}_{2}$, and $\mathbf{u}_{1}$ are ms$^{-1}$, m$^{2}$s$^{-1}$, and m$^{3}$s$^{-1}$, respectively.
{In addition, $\delta_{3}=1$, $\delta_{2}$ is the thickness of a fracture, and $\delta_{1}$ is the cross-section of a 1D preferential channel.}
The effective fluid source $f_{2}$
on $\Omega_{2}$ is given as
\begin{equation}
\label{eq:comm-term-2}
f_{2}=\delta_{2}\tilde{f}_{2}+\mathbf{u}_{3}^{+}\cdot\mathbf{n}^{+}%
+\mathbf{u}_{3}^{-}\cdot\mathbf{n}^{-},
\end{equation}
where 
$\tilde{f}_{2}$ is the
density of external fluid sources, and the normal fluxes from the two faces of the
3D continuum surrounding the fracture are given through the Robin (also called
Newton) boundary conditions
\begin{align}
\mathbf{u}_{3}^{+}\cdot\mathbf{n}^{+}  &  =\sigma_{3}^{+}(p_{3}^{+}%
-p_{2}),\label{eq:newton-1}\\
\mathbf{u}_{3}^{-}\cdot\mathbf{n}^{-}  &  =\sigma_{3}^{-}(p_{3}^{-}-p_{2}).
\label{eq:newton-2}%
\end{align}
In the last formula, $\sigma_{3}^{+/-}>0$ are the transition coefficients (cf.
\cite{Martin-2005-MFB} for possible choices) and $p_{3}^{+}$, $p_{3}^{-}$ are
the traces of pressure $p_{3}$ on the two sides of the fracture. The effective
fluid source $f_{1}$ on $\Omega_{1}$ is similar,
\begin{equation}
\label{eq:comm-term-1}
f_{1}=\delta_{1}\tilde{f}_{1}+\sum_{k}\mathbf{u}_{2}^{k}\cdot\mathbf{n}^{k},
\end{equation}
where 
$\tilde{f}_{1}$ is the density of external fluid sources. In the 3D ambient
space, the 1D channel can be connected to $k$ faces of 2D fractures. Thus
\begin{equation}
\mathbf{u}_{2}^{k}\cdot\mathbf{n}^{k}=\sigma_{2}^{k}(p_{2}^{k}-p_{1})
\label{eq:newton-3}%
\end{equation}
is the normal flux from the connected fracture $k$, $\sigma_{2}^{k}>0$ is the transition
coefficient, and $p_{2}^{k}$ is the trace of pressure $p_{2}$ on the face of
fracture $k$.

In the following, we describe the discrete mixed-hybrid formulation of the
problem. The formulation and discussion of the continuous problem can be found
in~\cite{Brezina-2010-MHF}. Let us consider spaces%
\begin{equation}
\mathbf{V}^{-1}=\mathbf{V}_{1}^{-1}\times\mathbf{V}_{2}^{-1}\times
\mathbf{V}_{3}^{-1},\ \mathbf{V}_{d}^{-1}=\prod_{i=1}^{N_{E}^{d}}%
\mathbf{V}^{i}(T_{d}^{i}),\quad Q=Q_{1}\times Q_{2}\times Q_{3},\ 
Q_{d}=L^{2}\left(  \Omega_{d}\right)  .
\end{equation}
For the definition of the space $\Lambda$, we cannot follow
(\ref{eq:lambda_space}) directly, since e.g. on $\Omega_{2}$, we need to
distinguish values of~$\lambda_{3}$ on two sides of a fracture. Thus, we
introduce a separate value for every non-boundary side of every element:
\begin{equation}
\Lambda(T_{d}^{i})=\left\{  \lambda\in L^{2}(\partial T_{d}^{i} \setminus \partial\Omega_d):\lambda
=\mathbf{v}\cdot\mathbf{n}|_{\partial T_{d}^{i}},\mathbf{v}\in\mathbf{V}%
_{d}\right\}  ,
\end{equation}
where $\mathbf{V}_{d}$ is defined in the same way as the space $\mathbf{V}$
but on the domain $\Omega_{d}$. Further, we identify values on faces/points
that are not aligned to the fractures/channels:
\begin{equation}
\Lambda_{d}=\Big\{\lambda\in\prod_{i=1}^{N_{E}^{d}}\Lambda(T_{d}%
^{i});\ \lambda|_{\partial T_{d}^{i}}=\lambda|_{\partial T_{d}^{j}}%
\quad\text{on face }F=\partial T_{d}^{i}\cap\partial T_{d}^{j}\quad\text{ if
}F\cap\Omega_{d-1}=\emptyset\Big\}.
\end{equation}
Finally, we redefine $\Lambda=\Lambda_{1}\times\Lambda_{2}\times\Lambda_{3}$.
In the \emph{mixed-hybrid finite element approximation} of the flow in
fractured porous media we seek a triple $\left\{  \mathbf{u},p,\lambda
\right\}  \in\mathbf{V}^{-1}\times Q\times\Lambda$ that satisfies
\begin{align}
a\left(  \mathbf{u},\mathbf{v}\right)  +b\left(  p,\mathbf{v}\right)
+b_{\mathcal{F}}\left(  \lambda,\mathbf{v}\right)   &  =\left\langle
g,\mathbf{v}\right\rangle ,\qquad\forall\mathbf{v\in V}^{-1}%
,\label{eq:hybrid-frac-1}\\
b\left(  \mathbf{u},q\right)  -c\left(  p,q\right)  -c_{\mathcal{F}%
}\left(  q,\lambda\right)   &  =\left\langle f,q\right\rangle
,\qquad\forall q\in Q,\label{eq:hybrid-frac-2}\\
b_{\mathcal{F}}\left(  \mathbf{u},\mu\right)  -c_{\mathcal{F}}\left(
p,\mu\right)  -\widetilde{c}\left(  \lambda,\mu\right)   &  =0,\qquad
\qquad\forall\mu\in\Lambda, \label{eq:hybrid-frac-3}
\end{align}
with
\begin{align}
a\left(  \mathbf{u},\mathbf{v}\right)   &  =\sum_{d=1}^{3}\sum_{i=1}%
^{N_{E}^{d}}\left[  \int_{T_{d}^{i}}\frac{1}{\delta_{d}}\Bbbk_{d}%
^{-1}\mathbf{u}_{d}\cdot\mathbf{v}_{d}\,dx\right]  ,\label{eq:weak_term_a}\\
b\left(  \mathbf{u},q\right)   &  =-\sum_{d=1}^{3}\sum_{i=1}^{N_{E}^{d}%
}\left[  \int_{T_{d}^{i}}q_{d}\,\left(  \nabla\cdot\mathbf{u}_{d}\right)
\,dx\right]  ,\\
b_{\mathcal{F}}\left(  \mathbf{u},\lambda\right)   &  =\sum_{d=1}^{3}%
\sum_{i=1}^{N_{E}^{d}}\left[  \int_{\partial T_{d}^{i}\setminus\partial
\Omega_{d}}\lambda|_{\partial T_{d}^{i}}\,\left(  \mathbf{u}_{d}%
\cdot\mathbf{n}\right)  \,ds\right]  ,\\
c\left(  p,q\right)   &  =\sum_{d=2}^{3}\sum_{i=1}^{N_{E}^{d}%
}\left[  \int_{\partial T_{d}^{i}\cap\Omega_{d-1}}\sigma_{d}\,\,p_{d-1}%
\,q_{d-1}\,ds\right]  ,\label{eq:weak_term_cbar}\\
c_{\mathcal{F}}\left(  p,\mu\right)   &  =-\sum_{d=2}^{3}\sum_{i=1}^{N_{E}%
^{d}}\left[  \int_{\partial T_{d}^{i}\cap\Omega_{d-1}}\sigma_{d}\ p_{d-1}%
\,\mu_{d}\,ds\right]  ,\label{eq:weak_term_cF}\\
\widetilde{c}\left(  \lambda,\mu\right)   &  =\sum_{d=2}^{3}\sum_{i=1}%
^{N_{E}^{d}}\left[  \int_{\partial T_{d}^{i}\cap\Omega_{d-1}}\sigma
_{d}\,\,\lambda_{d}\,\mu_{d}\,ds\right]  ,\label{eq:weak_term_ctilde}\\
\left\langle g,\mathbf{v}\right\rangle  &  =-\sum_{d=1}^{3}\sum_{i=1}%
^{N_{E}^{d}}\int_{\partial T_{d}^{i}\cap\partial\Omega_{N}}p_{N}\,\left(
\mathbf{v}\cdot\mathbf{n}\right)  \,ds 
{-\sum_{d=1}^{3}\sum_{i=1}^{N_{E}^{d}}\int_{T_{d}^{i}}v_{z}\,dx},\\
\left\langle f,q\right\rangle  &  =-\sum_{d=1}^{3}\int_{\Omega
}\delta_{d}\,\tilde{f}_{d}\,q_{d}\,dx. \label{eq:weak_term_f}%
\end{align}

The system~(\ref{eq:hybrid-frac-1})--(\ref{eq:hybrid-frac-3}) now leads to the
matrix form
\begin{equation}
\left[
\begin{array}
[c]{ccc}%
A & B^{T} & B_{\mathcal{F}}^{T}\\
B & -C & -C_{\mathcal{F}}^{T}\\
B_{\mathcal{F}} & -C_{\mathcal{F}} & -\widetilde{C}%
\end{array}
\right]  \left[
\begin{array}
[c]{c}%
\mathbf{u}\\
p\\
\lambda
\end{array}
\right]  =\left[
\begin{array}
[c]{c}%
g\\
f\\
0
\end{array}
\right]  . \label{eq:hybrid-m-fractured}%
\end{equation}

We note that~(\ref{eq:hybrid-m-fractured}) is related
to~(\ref{eq:hybrid-frac-1})--(\ref{eq:hybrid-frac-3}) in the same way
as~(\ref{eq:hybrid-m}) is related to~(\ref{eq:hybrid-1})--(\ref{eq:hybrid-3}).
The main difference in the structure of the matrices
between (\ref{eq:hybrid-m-fractured}) and~(\ref{eq:hybrid-m}) is the additional
block $\ol{C}=\left[
\begin{array}
[c]{cc}%
C & C_{\mathcal{F}}^{T}\\
C_{\mathcal{F}} & \widetilde{C}%
\end{array}
\right]  $ related to the normal fluxes between dimensions and arising
from~(\ref{eq:newton-1})--(\ref{eq:newton-2}) and (\ref{eq:newton-3}) via
(\ref{eq:weak_term_cbar})--(\ref{eq:weak_term_ctilde}).
{
In particular, 
the modified right-hand side of the continuity equation for two-dimensional and one-dimensional  
elements, $f_{2}$ and $f_{1}$, incorporates pressure unknowns on 2D and 1D  elements, 
and traces of pressure on 3D and 2D elements at the fracture, 
which are nothing but the Lagrange multipliers on 3D and 2D elements in the mixed-hybrid method.
Consequently, 
$p_{3}^{+/-} = \lambda_{3}^{+/-}$ in~(\ref{eq:newton-1})--(\ref{eq:newton-2}), and $p_{2}^{k} = \lambda_{2}^{k}$ in~(\ref{eq:newton-3}).
}

Assuming~$\delta_{d}$ bounded and greater than zero, and using the fact
that~$\Bbbk_{d}$ corresponds to a symmetric positive definite matrix, we see
from (\ref{eq:weak_term_a}) that block $A$ in~(\ref{eq:hybrid-m-fractured}) is
symmetric positive definite. Block $\ol{C}$ is symmetric positive semi-definite
since
\begin{equation}
\label{eq:C_semi_definite}c(p,p) + 2c_{\mathcal{F}}(p,\lambda) +
\widetilde{c}(\lambda, \lambda) = \sum_{d=2}^{3}\sum_{i=1}^{N_{E}^{d}}\left[
\int_{\partial T_{d}^{i} \cap\Omega_{d-1}} \sigma_{d} (p_{d-1} - \lambda
_{d})^{2} \,ds \right]  .
\end{equation}

The following theorem is a standard result, e.g. \cite[Theorem~1.2]%
{Brezzi-1991-MHF}. Here, we rewrite it in a form suitable for our setting and
we verify the assumptions for the specific blocks of the matrix in
(\ref{eq:hybrid-m-fractured}).
{
We will further denote $\ol{Q} = Q\times \Lambda$, 
$\ol{p}=(p,\lambda) \in \ol{Q}$, $\ol{q}=(q, \mu) \in \ol{Q}$, and 
$\ol{b}(\vc u, \ol{q})=b(\vc u, q) + b_{\mathcal{F}}(\vc u, \mu)$. 
}

\begin{theorem}
\label{th:regular} Let natural boundary conditions (\ref{eq:problem-3}) be
prescribed at a certain part of the boundary, i.e. $\partial\Omega_{N,d}
\neq\emptyset$ for at least one $d\in\{1,2,3\}$. Then the discrete
mixed-hybrid problem~(\ref{eq:hybrid-m-fractured})
has a unique solution.
\end{theorem}

\def\component{\Omega^c_k}
\begin{proof}
{
Let us first investigate the structure of the matrix in 
(\ref{eq:hybrid-m-fractured}) more closely. 
Let us number the unknowns within each block of (\ref{eq:hybrid-m-fractured}) 
with respect to spatial dimension $d\in\{1,2,3\}$. The matrix then takes the 
form of 9$\times $9 blocks,
\begin{equation}
\label{eq:crazy_blocks_matrix}
\left[
\begin{matrix}
A_{11}             &                    &                    & B^T_{11}            &                     &            & B^T_{\mathcal{F},11}  &                       &                       \\
                   & A_{22}             &                    &                     & B^T_{22}            &            &                       & B^T_{\mathcal{F},22}  &                       \\
                   &                    & A_{33}             &                     &                     &  B^T_{33}  &                       &                       & B^T_{\mathcal{F},33}  \\
B_{11}             &                    &                    & -C_{11}             &                     &            &                       & -C^T_{\mathcal{F},12} &                       \\
                   & B_{22}             &                    &                     & -C_{22}             &            &                       &                       & -C^T_{\mathcal{F},23} \\
                   &                    & B_{33}             &                     &                     &            &                       &                       &                       \\
B_{\mathcal{F},11} &                    &                    &                     &                     &            &                       &                       &                       \\
                   & B_{\mathcal{F},22} &                    & -C_{\mathcal{F},12} &                     &            &                       & -\widetilde{C}_{22}   &                       \\
                   &                    & B_{\mathcal{F},33} &                     & -C_{\mathcal{F},23} &            &                       &                       & -\widetilde{C}_{33}   \\
\end{matrix}
\right].
\end{equation}
Suppose for a~moment that we solve a~problem only on domain $\Omega_{d}$, $d \in {1,2,3}$ (i.e. $\Omega_{i} = \emptyset$ for $i\ne d$). 
If no {natural} boundary conditions are imposed, there is a single 
$-1$ entry on each row of $B^T_{dd}$ and a single $+1$ entry on each row of 
$B^T_{\mathcal{F},dd}$.
Since $\Omega_{d}$ is a simply connected set, 
the matrix
$\ol{B}^T_{dd} = \left[
\begin{smallmatrix}
B^T_{dd} & B^T_{\mathcal{F},dd}
\end{smallmatrix}
\right]$
has a nontrivial nullspace of constant vectors.
Enforcing {natural} boundary condition on a part of $\Omega_d$ changes the
$+1$ value on the corresponding row of matrix $B^T_{\mathcal{F},dd}$ to $0$, 
in which case $\ol{B}^T_{dd}$ has only a trivial nullspace, i.e. full column 
rank (see e.g. \cite[Section IV.1]{Brezzi-1991-MHF} or 
\cite[Lemma~3.2]{Maryska-1995-MHF}).

The nullspace becomes more complicated for domains with fractures, in which 
case $\Omega_{d}$ typically has more simply connected components 
separated by fractures (cf. Fig.~\ref{fig:schemes}). Let us denote them 
$\component$, 
$k=1,\dots, n_c$, regardless of the dimension. In particular, $\component$, 
$k=1,\dots, n_{ci}$ will be components without natural condition boundary, i.e.
$\partial \component \cap \partial \Omega_{N}=\emptyset$, while for 
$k=n_{ci}+1, 
\dots, n_c$ we get components with prescribed natural boundary condition.
We also denote
$\ol\chi_k\in\ol{Q}$ the characteristic vector of the component $\component$ 
that takes value $+1$  for degrees of freedom associated with elements and faces
of the component $\component$. With such notation, the basis of the 
nullspace of the whole matrix
\begin{equation}
\label{eq:blocks_matrix}
\ol{B}^T =
\left[
\begin{matrix}
B^T_{11}            &                     &            & B^T_{\mathcal{F},11}  & 
                      &                       \\
                    & B^T_{22}            &            &                       & 
B^T_{\mathcal{F},22}  &                       \\
                    &                     &  B^T_{33}  &                       & 
                      & B^T_{\mathcal{F},33}  \\
\end{matrix}
\right]
\end{equation}
consists of characteristic vectors $\ol{\chi}_k$, $k=1,\dots n_{ci}$ and has 
dimension $n_{ci}$.

Next, we show that matrix $\ol{C}$ is not only symmetric positive 
semi-definite, as seen from (\ref{eq:C_semi_definite}), but also positive definite on 
$\text{null}\,\ol{B}^T$. To this end, take $\ol{p} \in 
\mbox{null}\,\ol{B}^T$, a vector  
that is piecewise constant on components, having value $\ol{p}_{k,d}$ on the 
component $\component$ of dimension $d$ for $k=1,\dots, n_{ci}$, and value 
$\ol{p}_{k,d}=0$ for other components. 
Then $\ol{p}^T \ol{C}\ol{p}=0$ implies 
$\ol{p}=\ol{p}_{k,d}=0$. 
Indeed, every component $\component$ of dimension $d=2,3$ 
has some component $\Omega^c_j$ of dimension $d-1$ on its boundary,
and therefore all $\ol{p}_{k,d}$ have the same value, cf. (\ref{eq:C_semi_definite}). 
This value is zero, since there is at least one component with natural boundary condition.

Applying the congruence transformation,
\begin{equation}
\label{eq:congruence1}
\left[
\begin{matrix}
A        & \ol{B}^T  \\
\ol{B}   & -\ol{C}
\end{matrix}
\right]
=
\left[
\begin{matrix}
I                 &   \\
\ol{B} A^{-1} & I
\end{matrix}
\right]
\left[
\begin{matrix}
A &   \\
& -\left( \ol{B}A^{-1}\ol{B}^T + \ol{C} \right)
\end{matrix}
\right]
\left[
\begin{matrix}
I    & A^{-1} \ol{B}^T \\
& I
\end{matrix}
\right].
\end{equation}
Matrix $A$ is symmetric positive definite (SPD) from (\ref{eq:weak_term_a})
and therefore  $\ol{B}A^{-1}\ol{B}^T$ is SPD on range $\ol{B}$, which is 
the orthogonal complement of the nullspace of $\ol{B}^T$.
Thus $\ol{B}A^{-1}\ol{B}^T + \ol{C}$ is SPD on whole $\ol{Q}$.
From the Sylvester law of inertia, the number of positive, negative and zero 
eigenvalues is preserved by the congruence transformation.
Since the block diagonal matrix on the right-hand side of~(\ref{eq:congruence1}) has only (strictly) positive and (strictly) negative eigenvalues,
the matrix on the left-hand side also must be nonsingular,
and problem~(\ref{eq:hybrid-m-fractured}) has a unique solution.

}
\end{proof}

\begin{figure}[ptbh]
\begin{center}
\includegraphics[width=0.25\textwidth]{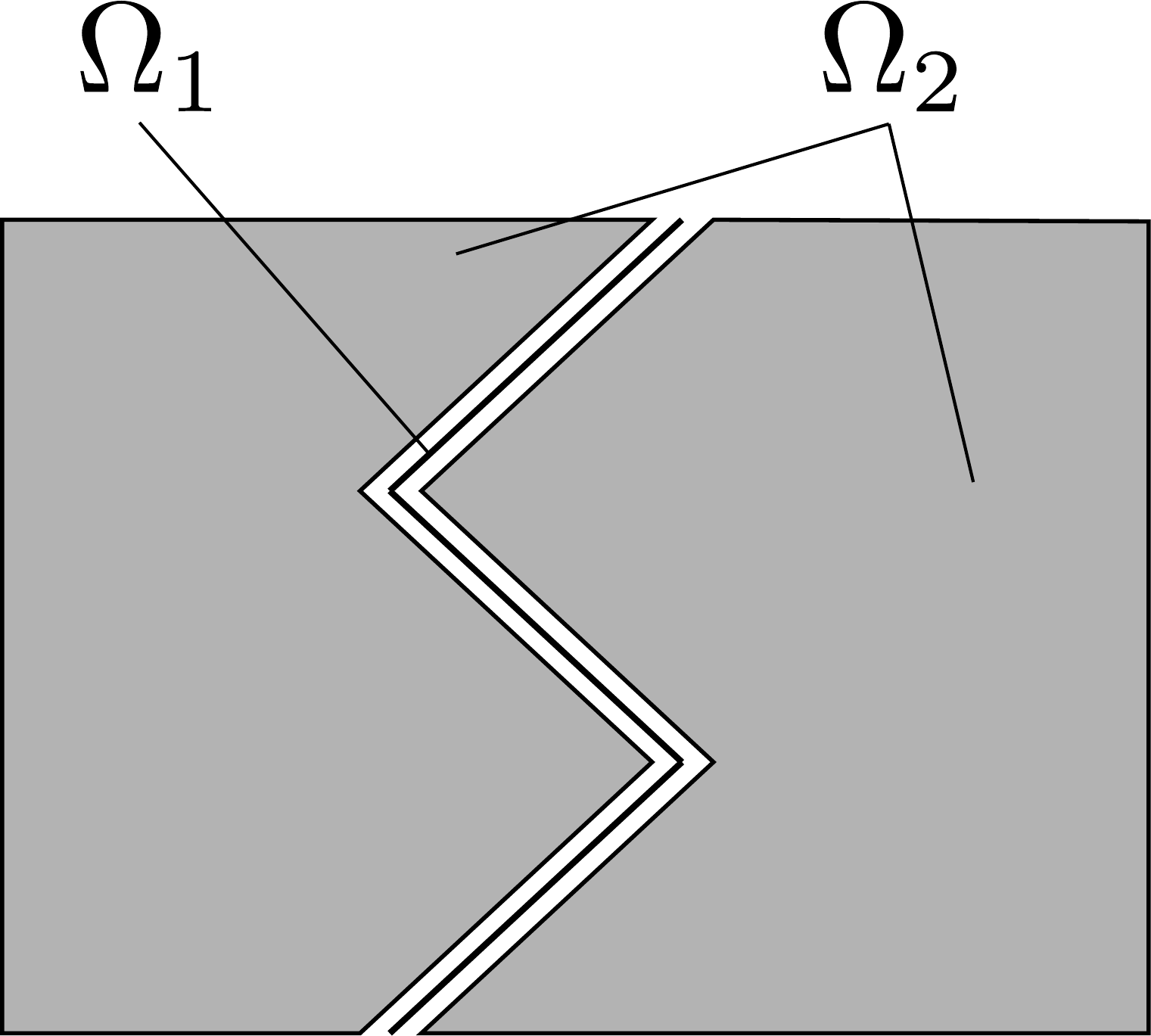} \hspace{25mm} 
\includegraphics[width=0.30\textwidth]{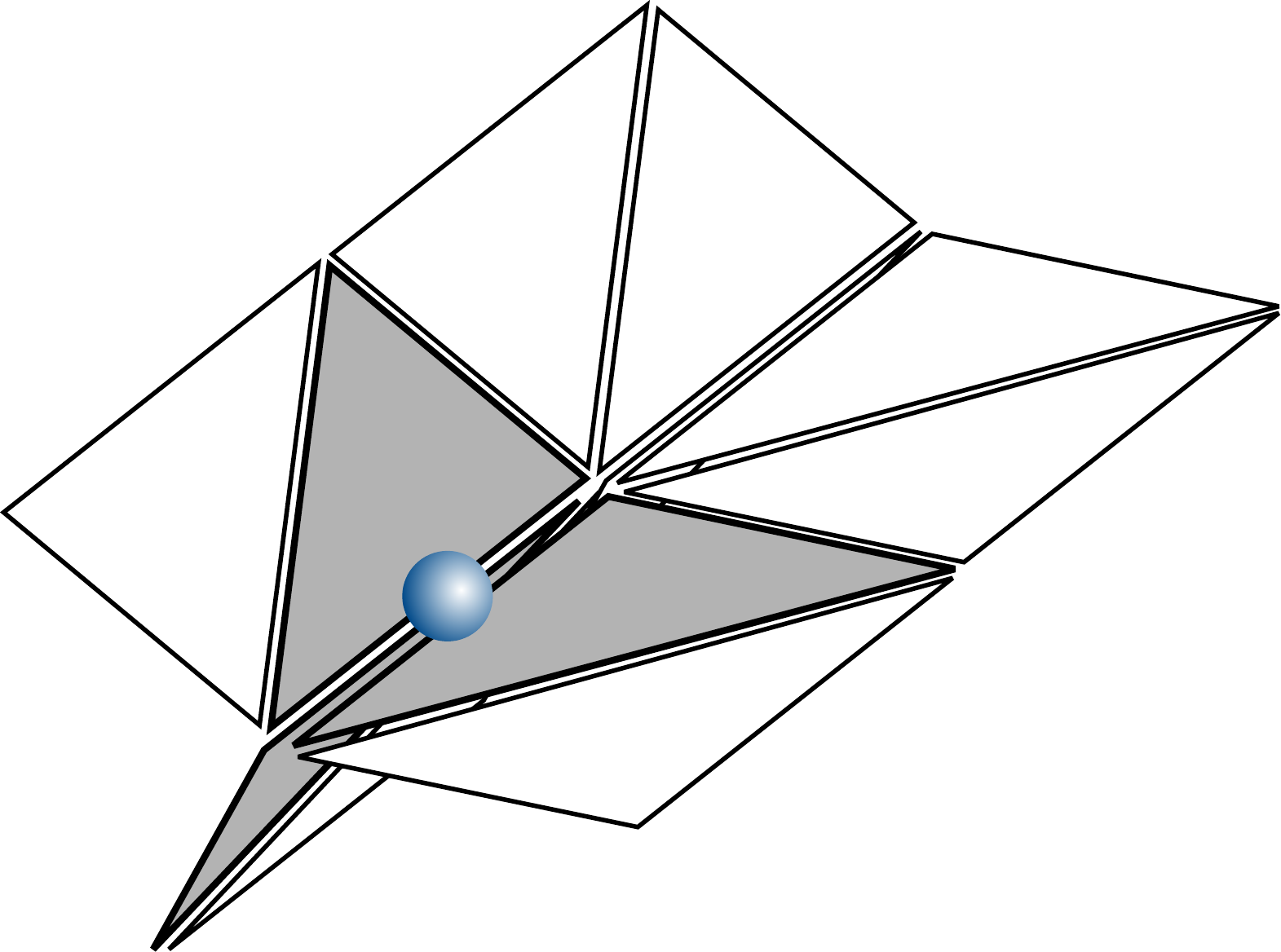}
\end{center}
\caption{\label{fig:schemes}
Example of a two-dimensional problem: even single fracture gives rise to two components of the two-dimensional mesh (left);
example of three (shaded) triangles sharing a~single Lagrange multiplier in 3D (right).}
\end{figure}

\section{Iterative substructuring}

\label{sec:substructuring}

For our purposes of combining meshes with different spatial dimensions, we
define \emph{substructures} as subsets of finite elements in the mesh rather
than parts of a physical domain, cf.~\cite{Toselli-2005-DDM}.

To begin, let us define the combined triangulation $\mathcal{T}_{123}$ as the
union of triangulations for each spatial dimension, i.e., $\mathcal{T}%
_{123}=\mathcal{T}_{1}\cup\mathcal{T}_{2}\cup\mathcal{T}_{3}$. The
triangulation $\mathcal{T}_{123}$ is subsequently divided into
substructures$~\Omega^{i},\;i=1,\dots,N_{S}$. Note that in general a
substructure can contain finite elements of different dimensions. We
define the interface $\Gamma$ as the set of degrees of freedom shared by more
than one substructure. 
Note that $\Gamma \subset \Lambda$ for the current setting.
Thus, let us split the Lagrange multipliers$~\lambda$ into
two subsets. First, we denote by$~\lambda_{\Gamma}$ the set of multipliers
corresponding to the interface$~\Gamma$. The remaining multipliers,
corresponding to substructure interiors, will be denoted by$~\lambda_{I}$. The
interface$~\Gamma^{i}$ of substructure$~\Omega^{i}$ is defined as a subset
of$~\Gamma$ corresponding to$~\partial\Omega^{i}$. Next, let $\Lambda_{\Gamma
}^{i}$ be defined as the space of Lagrange multipliers corresponding
to~$\Gamma^{i}$, $i=1,\ldots,N_{S}$, and define a space
\begin{equation}
\Lambda_{\Gamma}=\Lambda_{\Gamma}^{1}\times\cdots\times\Lambda_{\Gamma}%
^{N_{S}}.
\end{equation}
The substructure problems are obtained by assembling contributions of finite
elements in each$~\Omega^{i}$,%
\begin{equation}
\left[
\begin{array}
[c]{cccc}%
A^{i} & B^{iT} & B_{\mathcal{F},I}^{iT} & B_{\mathcal{F},\Gamma}^{iT}\\
B^{i} & -C^{i} & -C_{\mathcal{F},I}^{iT} & -C_{\mathcal{F},\Gamma
}^{iT}\\
B_{\mathcal{F},I}^{i} & -C_{\mathcal{F},I}^{i} & -\widetilde{C}_{II}^{i} &
-\widetilde{C}_{{\Gamma}I}^{iT}\\
B_{\mathcal{F},\Gamma}^{i} & -C_{\mathcal{F},\Gamma}^{i} & -\widetilde{C}%
_{{\Gamma}I}^{i} & -\widetilde{C}_{{\Gamma}{\Gamma}}^{i}%
\end{array}
\right]  \left[
\begin{array}
[c]{c}%
\mathbf{u}^{i}\\
p^{i}\\
\lambda_{I}^{i}\\
\lambda_{\Gamma}^{i}%
\end{array}
\right]  =\left[
\begin{array}
[c]{c}%
g^{i}\\
f^{i}\\
0\\
0
\end{array}
\right]  ,\qquad i=1,\dots,N_{S}, \label{eq:hybrid-fractured-substructure}%
\end{equation}
where the blocks $A^{i}$ are block diagonal with blocks corresponding to
finite element matrices, and the blocks $\ol{C}^{i}\neq0$ only if the
substructure $\Omega^{i}$ contains coupling of elements of different
dimensions. We note that the global problem (\ref{eq:hybrid-m-fractured}) could
be obtained from (\ref{eq:hybrid-fractured-substructure}) by further assembly at the interface.

In \emph{the iterative substructuring} (see e.g.~\cite{Toselli-2005-DDM}), we first
reduce the problem to substructure interfaces. In our context we can eliminate
normal fluxes, pressure unknowns, and Lagrange multipliers at interiors of
substructures, and we can define the substructure Schur complements$~S^{i}%
:\Lambda_{\Gamma}^{i}\mapsto\Lambda_{\Gamma}^{i}$, $i=1,\dots,N_{S}$, {formally} as
{
\begin{equation}
S^{i} = 
\widetilde{C}_{{\Gamma}{\Gamma}}^{i} 
+ 
\left[
\begin{array}
[c]{ccc}%
B_{\mathcal{F},\Gamma}^{i} & -C_{\mathcal{F},\Gamma}^{i} & -\widetilde{C}_{{\Gamma}I}^{i}%
\end{array}
\right]  
\left[
\begin{array}
[c]{ccc}%
A^{i} & B^{iT} & B_{\mathcal{F},I}^{iT}\\
B^{i} & -C^{i} & -C_{\mathcal{F},I}^{iT}\\
B_{\mathcal{F},I}^{i} & -C_{\mathcal{F},I}^{i} & -\widetilde{C}_{II}^{i}%
\end{array}
\right]^{-1}
\left[
\begin{array}
[c]{c}%
B_{\mathcal{F},\Gamma}^{iT}\\
-C_{\mathcal{F},\Gamma}^{iT}\\
-\widetilde{C}_{{\Gamma}I}^{iT}%
\end{array}
\right] 
. 
\label{eq:substructure-schur}%
\end{equation}
}
{However, in an implementation of a Krylov subspace iterative method, 
we only need to compute the matrix-vector
product $S^{i}\lambda_{\Gamma}^{i}$ for a given vector$~\lambda_{\Gamma}^{i}$.
Therefore, the matrix is not constructed explicitly, and  
the multiplication is obtained as follows.
}

\begin{algorithm}
\label{alg:S-mult} Given $\lambda_{\Gamma}^{i}\in\Lambda_{\Gamma}^{i}$,
determine $S^{i}\lambda_{\Gamma}^{i}\in\Lambda_{\Gamma}^{i}$ in the following
two steps:

\begin{enumerate}
\item solve a local \emph{Dirichlet problem}
\begin{equation}
\left[
\begin{array}
[c]{ccc}%
A^{i} & B^{iT} & B_{\mathcal{F},I}^{iT}\\
B^{i} & -C^{i} & -C_{\mathcal{F},I}^{iT}\\
B_{\mathcal{F},I}^{i} & -C_{\mathcal{F},I}^{i} & -\widetilde{C}_{II}^{i}%
\end{array}
\right]  \left[
\begin{array}
[c]{c}%
\mathbf{w}^{i}\\
q^{i}\\
\mu_{I}^{i}%
\end{array}
\right]  =-\left[
\begin{array}
[c]{c}%
B_{\mathcal{F},\Gamma}^{iT}\\
-C_{\mathcal{F},\Gamma}^{iT}\\
-\widetilde{C}_{{\Gamma}I}^{iT}%
\end{array}
\right]  \lambda_{\Gamma}^{i},
\label{eq:discrete-dirichlet-problem-substructure}%
\end{equation}

\item perform two sparse matrix-vector multiplications
\begin{equation}
S^{i}\lambda_{\Gamma}^{i}\longleftarrow - \left( - \widetilde{C}_{{\Gamma}{\Gamma}}%
^{i}\lambda_{\Gamma}^{i} + \left[
\begin{array}
[c]{ccc}%
B_{\mathcal{F},\Gamma}^{i} & -C_{\mathcal{F},\Gamma}^{i} & -\widetilde{C}%
_{{\Gamma}I}^{i}%
\end{array}
\right]  \left[
\begin{array}
[c]{c}%
\mathbf{w}^{i}\\
q^{i}\\
\mu_{I}^{i}%
\end{array}
\right] 
\right). 
\label{eq:substructure-schur-multiple}%
\end{equation}

\end{enumerate}
\end{algorithm}

Next, let $\widehat{\Lambda}_{\Gamma}$ be the space of global degrees of
freedom, such that the values of degrees of freedom from $\Lambda_{\Gamma}$
coincide on$~\Gamma$. The vectors of the local substructure degrees of freedom
$\lambda_{\Gamma}^{i}\in\Lambda_{\Gamma}^{i}$ and the vector of the global
degrees of freedom $\lambda_{\Gamma}\in\widehat{\Lambda}_{\Gamma}$ are related
by
\begin{equation}
\lambda_{\Gamma}^{i}=R^{i}\lambda_{\Gamma},\quad i=1,\dots,N_{S},\label{eq:R}%
\end{equation}
where $R^{i}$ are the restriction operators. More specifically, each $R^{i}$ is a $0$-$1$
matrix such that every row contains one entry {equal to one}, and $R^{i}R^{iT}=I$.
The global Schur complement $\widehat{S}:\widehat{\Lambda}_{\Gamma}%
\rightarrow\widehat{\Lambda}_{\Gamma}$ can be obtained as
\begin{equation}
\widehat{S}=\sum_{i=1}^{N_{S}}R^{iT}S^{i}R^{i}.\label{eq:schur-assembly}%
\end{equation}
Equation (\ref{eq:schur-assembly}) represents the formal assembly of
substructure Schur complements into the global Schur complement. The global
Schur complement system, which we would like to solve iteratively, reads
\begin{equation}
\widehat{S}\lambda_{\Gamma}=\widehat{b},\label{eq:S_hat-system}%
\end{equation}
where vector $\widehat{b}=\sum_{i=1}^{N_{S}}R^{iT}b^{i}$ is obtained from
substructure reduced right-hand sides
\begin{equation}
b^{i}=\left[
\begin{array}
[c]{ccc}%
B_{\mathcal{F},\Gamma}^{i} & -C_{\mathcal{F},\Gamma}^{i} & -\widetilde{C}%
_{{\Gamma}I}^{i}%
\end{array}
\right]  \left[
\begin{array}
[c]{ccc}%
A^{i} & B^{iT} & B_{\mathcal{F},I}^{iT}\\
B^{i} & -C^{i} & -C_{\mathcal{F},I}^{iT}\\
B_{\mathcal{F},I}^{i} & -C_{\mathcal{F},I}^{i} & -\widetilde{C}_{II}^{i}%
\end{array}
\right]  ^{-1}\left[
\begin{array}
[c]{c}%
g^{i}\\
f^{i}\\
0
\end{array}
\right]  .\label{eq:reduced-rhs}%
\end{equation}
In our implementation, we factor and store the matrices
from~(\ref{eq:discrete-dirichlet-problem-substructure}). The factors are then
used to compute the vectors$~b^{i}$ in~(\ref{eq:reduced-rhs}), and especially
in Algorithm~\ref{alg:S-mult}, which is used in connection to formula~(\ref{eq:schur-assembly}) 
to evaluate $\lambda_{\Gamma}\rightarrow
\widehat{S}\lambda_{\Gamma}$ within each iteration of a Krylov subspace iterative
method. This algorithm allows a straightforward parallel implementation. After an approximate
solution of~(\ref{eq:S_hat-system}) is found, solution in interiors of
substructures, including all primal variables, is recovered
from (\ref{eq:hybrid-fractured-substructure})
using the factors
from~(\ref{eq:discrete-dirichlet-problem-substructure}).

\begin{remark}
There are other ways to derive the interface problem~(\ref{eq:S_hat-system}).
The authors of~\cite{Tu-2007-BAF,Cowsar-1995-BDD} use a mixed-hybrid
formulation with Lagrange multipliers introduced only at the interface
$\Gamma$ as their starting point. 
{While problem~(\ref{eq:S_hat-system}) is equivalent to the interface problems 
considered in~\cite{Tu-2007-BAF,Cowsar-1995-BDD}, 
the substructure problems therein have a different structure
from~(\ref{eq:hybrid-fractured-substructure}).} 
In particular, there are no
blocks corresponding to~$\lambda_{I}^{i}$, and the matrices corresponding
to~$A^{i}$ are no longer block diagonal element-wise. Next, the authors
of~\cite{Maryska-2000-SCS,Brezina-2010-MHF} construct explicit Schur
complement with respect to the whole block of Lagrange multipliers $\lambda$
and they show that, due to the special structure of~$A$, the complement is
both sparse and reasonably cheap to construct. If this was performed
substructure by substructure, this could be again seen as an intermediate step
in obtaining problem~(\ref{eq:S_hat-system}) by additional elimination of the
interior Lagrange multipliers~$\lambda_{I}$. However, this would again lead to
different substructure problems based on explicit local Schur complements.
\end{remark}

The following result allows an application of the BDDC\ method to problems
with fractures.

\begin{theorem}
\label{th:sspd} Let natural boundary conditions (\ref{eq:problem-3}) be
prescribed at a certain part of the boundary, i.e. $\partial\Omega_{N,d}
\neq\emptyset$ for at least one $d\in\{1,2,3\}$. Then the matrix~$\widehat{S}$
in (\ref{eq:S_hat-system}) is symmetric positive definite.
\end{theorem}

\begin{proof}
Using the notation of~(\ref{eq:blocks_matrix})--(\ref{eq:congruence1}),
let us introduce a matrix $\mathcal{S} = \ol{B}A^{-1}\ol{B}^T + \ol{C}$.
The matrix~$\mathcal{S}$ is symmetric positive definite by Theorem~\ref{th:regular} (and its proof).
Applying another congruence transformation to $\mathcal{S}$
and denoting the rows corresponding to the interface Lagrange multipliers  by subscript $_{\Gamma}$
and the interior by~$_{I}$,
we obtain
\begin{equation}
\label{eq:proof}
\mathcal{S}
=
\left[
\begin{matrix}
\mathcal{S}_{II}       & \mathcal{S}_{\Gamma I}^T  \\
\mathcal{S}_{\Gamma I} & \mathcal{S}_{\Gamma \Gamma}
\end{matrix}
\right]
=
\left[
\begin{matrix}
I                                          &   \\
\mathcal{S}_{\Gamma I} \mathcal{S}_{II}^{-1} & I
\end{matrix}
\right]
\left[
\begin{matrix}
\mathcal{S}_{II} &   \\
& \mathcal{S}_{\Gamma \Gamma} - \mathcal{S}_{\Gamma I} \mathcal{S}_{II}^{-1} \mathcal{S}_{\Gamma I}^{T}
\end{matrix}
\right]
\left[
\begin{matrix}
I    & \mathcal{S}_{II}^{-1} \mathcal{S}_{\Gamma I}^{T}  \\
& I
\end{matrix}
\right].
\end{equation}
Since the matrix on the left-hand side is SPD,
both diagonal blocks $\mathcal{S}_{II}$
and $\mathcal{S}_{\Gamma \Gamma} - \mathcal{S}_{\Gamma I} \mathcal{S}_{II}^{-1} \mathcal{S}_{\Gamma I}^{T}$
are also symmetric positive definite from the Sylvester law of inertia.
It remains to note that the Schur complement $\widehat{S}$ in (\ref{eq:S_hat-system}) is symmetric positive definite
because
$\widehat{S} = \mathcal{S}_{\Gamma \Gamma} - \mathcal{S}_{\Gamma I} \mathcal{S}_{II}^{-1} \mathcal{S}_{\Gamma I}^{T}$.
\end{proof}

Theorem~\ref{th:sspd} allows us to use the conjugate gradient method for the iterative solution 
of~(\ref{eq:S_hat-system}). 
In the next section, we describe the BDDC method used as {a preconditioner for $\widehat{S}$}.

\section{The BDDC preconditioner}

\label{sec:bddc}

In this section, we formulate the BDDC method for the solution
of~(\ref{eq:S_hat-system}). The algorithm can be viewed as a generalisation
of~\cite{Tu-2007-BAF}. However, we follow the original description from~\cite{Dohrmann-2003-PSC}, 
which better reflects our implementation. 

One step of BDDC provides a two-level preconditioner for the conjugate gradient method applied to solving problem (\ref{eq:S_hat-system}).
It is characterised by the selection of certain \emph{coarse degrees of freedom} based on primary degrees of freedom at interface $\Gamma$. 
The main coarse degrees of freedom in this paper are arithmetic averages over \emph{faces}, 
defined as subsets of degrees of freedom shared by the same two substructures. 
In addition, \emph{corner} coarse degrees of freedom,
defined as any selected Lagrange multiplier at the interface, are used.
Substructure \emph{edges}, defined as subsets of degrees of freedom shared by several substructures, may also appear (see Remark~\ref{rem:constraints}). 

The BDDC method introduces \emph{constraints} which enforce continuity of functions from ${\Lambda}_{\Gamma}$ 
at coarse degrees of freedom among substructures.
This gives raise to the space $\widetilde{\Lambda}_{\Gamma}$, 
which is given as the subspace of ${\Lambda}_{\Gamma}$ of functions 
which satisfy these continuity constraints. 
In particular,
\begin{equation}
\label{eq:subspaces}
\widehat{\Lambda}_{\Gamma} \subset \widetilde{\Lambda}_{\Gamma} \subset {\Lambda}_{\Gamma}.
\end{equation}

\begin{remark}
\label{rem:constraints}
{In three spatial dimensions, 
several triangular elements can be connected at a single Lagrange multiplier 
in a star-like configuration (cf. Fig.~\ref{fig:schemes}).
A similar statement holds for line elements considered in two- or three-dimensional space.
This fact may lead to the presence of substructure \emph{edges} and even \emph{vertices} 
(defined as degenerate edges consisting of a single degree of freedom),
and we may prescribe also edge averages as constraints. 
As mentioned above, we
also select \emph{corners} as coarse degrees of freedom. While essentially any degree
of freedom at the interface~$\Gamma$ can be \emph{a corner}, we select them by the
face-based algorithm~\cite{Sistek-2012-FSC}. 
This algorithm considers all vertices as corners, and, in addition, 
it selects three geometrically well distributed degrees of freedom 
from the interface between two substructures sharing a face into the set of \emph{corners}.
Although considering corners is not the standard
practice with $RT_{0}$ finite elements, in our experience, corners improve convergence
for numerically difficult problems, as can be observed for the engineering applications
presented in Section~\ref{sec:numerical}.}
\end{remark}

We now proceed to the formulation of
operators used in the BDDC method. The choice of constraints determines the
construction of matrices~$D^{i}$. Each row of$~D^{i}$ defines one coarse
degree of freedom at substructure~$\Omega^{i}$, e.g., a~corner
corresponds to a single $1$ entry at a row and an arithmetic average to
several $1$'s at a row. The \emph{coarse basis functions}~$\Phi_{\Gamma}^{i}$,
one per each substructure coarse degree of freedom, are computed by augmenting
the matrices from~(\ref{eq:hybrid-fractured-substructure}) with~$D^{i}$, and
solving the augmented systems with multiple right-hand sides 
\begin{equation}
\left[
\begin{array}
[c]{ccccc}%
A^{i} & B^{iT} & B_{\mathcal{F},I}^{iT} & B_{\mathcal{F},\Gamma}^{iT} & 0\\
B^{i} & -C^{i} & -C_{\mathcal{F},I}^{iT} & -C_{\mathcal{F},\Gamma
}^{iT} & 0\\
B_{\mathcal{F},I}^{i} & -C_{\mathcal{F},I}^{i} & -\widetilde{C}_{II}^{i} &
-\widetilde{C}_{{\Gamma}I}^{iT} & 0\\
B_{\mathcal{F},\Gamma}^{i} & -C_{\mathcal{F},\Gamma}^{i} & -\widetilde{C}%
_{{\Gamma}I}^{i} & -\widetilde{C}_{{\Gamma}{\Gamma}}^{i} & D^{iT}\\
0 & 0 & 0 & D^{i} & 0
\end{array}
\right]  \left[
\begin{array}
[c]{c}%
X^{i}\\
Z^{i}\\
\Phi_{I}^{i}\\
\Phi_{\Gamma}^{i}\\
L^{i}%
\end{array}
\right]  =\left[
\begin{array}
[c]{c}%
0\\
0\\
0\\
0\\
I
\end{array}
\right]  ,\quad i=1,\dots,N_{S},
\label{eq:hybrid-fractured-augmented-substructure}%
\end{equation}
where $I$ is the identity matrix, and $X^{i}$, $Z^{i}$, and $\Phi_{I}^{i}$ are
auxiliary matrices not used any further. 
As shown in \cite{Pultarova-2012-PCP},
the \emph{local coarse matrix}~$S_{CC}^{i}$ is obtained 
as a side product of solving~(\ref{eq:hybrid-fractured-augmented-substructure}) 
as
\begin{equation}
S_{CC}^{i}=\left[
X^{iT}\ Z^{iT}\ \Phi_{I}^{iT}\ \Phi_{\Gamma}^{iT}%
\right]  \left[
\begin{array}
[c]{cccc}%
A^{i} & B^{iT} & B_{\mathcal{F},I}^{iT} & B_{\mathcal{F},\Gamma}^{iT}\\
B^{i} & -C^{i} & -C_{\mathcal{F},I}^{iT} & -C_{\mathcal{F},\Gamma
}^{iT}\\
B_{\mathcal{F},I}^{i} & -C_{\mathcal{F},I}^{i} & -\widetilde{C}_{II}^{i} &
-\widetilde{C}_{{\Gamma}I}^{iT}\\
B_{\mathcal{F},\Gamma}^{i} & -C_{\mathcal{F},\Gamma}^{i} & -\widetilde{C}%
_{{\Gamma}I}^{i} & -\widetilde{C}_{{\Gamma}{\Gamma}}^{i}
%&  &  &
\end{array}
\right]  \left[
\begin{array}
[c]{c}%
X^{i}\\
Z^{i}\\
\Phi_{I}^{i}\\
\Phi_{\Gamma}^{i}%
\end{array}
\right]  =-L^{i}. \label{eq:scci_def}%
\end{equation}
Let us define, similarly to~(\ref{eq:R}), operators$~R_{C}^{i}$\ that relate
vectors of local coarse degrees of freedom$~\mu_{C}^{i}$ to
the vector of global coarse degrees of freedom$~\mu_{C}$ as
\begin{equation}
\mu_{C}^{i}=R_{C}^{i}\mu_{C}.
\end{equation}
The \emph{global coarse matrix} $S_{CC}$ is then obtained by assembling the
local contributions as
\begin{equation}
S_{CC}=\sum_{i=1}^{N_{S}}R_{C}^{iT}S_{CC}^{i}R_{C}^{i}. \label{eq:Scc_def}%
\end{equation}
Finally, let us define the scaling operators
\begin{equation}
W^{i}:{\Lambda}_{\Gamma}^{i}\rightarrow{\Lambda}_{\Gamma}^{i},\quad
i=1,\dots,N^{S}, \label{eq:Wi}%
\end{equation}
which are given as diagonal matrices of weights 
that satisfy
\begin{equation}
\sum_{i=1}^{N_{S}}R^{iT}W^{i}R^{i}=I.
\end{equation}
More details on the selection of diagonal entries in $W^{i}$ are given in Section~\ref{sec:scaling}.

With this selection of spaces and operators, we are ready to formulate the BDDC preconditioner.
\begin{algorithm}
\label{alg:bddc}The BDDC preconditioner $M_{BDDC}:r_{\Gamma}\in
\widehat{\Lambda}_{\Gamma}\rightarrow\lambda_{\Gamma}\in\widehat{\Lambda
}_{\Gamma}$ is defined in the following steps:

\begin{enumerate}
\item Compute the local residuals
\begin{equation}
r_{\Gamma}^{i}=W^{i}R^{i}r_{\Gamma},\quad i=1,\dots,N_{S}.
\end{equation}

\item Compute the \emph{substructure corrections} $\eta_{\Gamma \Delta}^{i}$ by
solving the local\emph{ Neumann problems}
\begin{equation}
\left[
\begin{array}
[c]{ccccc}%
A^{i} & B^{iT} & B_{\mathcal{F},I}^{iT} & B_{\mathcal{F},\Gamma}^{iT} & 0\\
B^{i} & -C^{i} & -C_{\mathcal{F},I}^{iT} & -C_{\mathcal{F},\Gamma
}^{iT} & 0\\
B_{\mathcal{F},I}^{i} & -C_{\mathcal{F},I}^{i} & -\widetilde{C}_{II}^{i} &
-\widetilde{C}_{{\Gamma}I}^{iT} & 0\\
B_{\mathcal{F},\Gamma}^{i} & -C_{\mathcal{F},\Gamma}^{i} & -\widetilde{C}%
_{{\Gamma}I}^{i} & -\widetilde{C}_{{\Gamma}{\Gamma}}^{i} & D^{iT}\\
0 & 0 & 0 & D^{i} & 0
\end{array}
\right]  \left[
\begin{array}
[c]{c}%
x^{i}\\
z^{i}\\
\eta_{I \Delta }^{i}\\
\eta_{\Gamma \Delta}^{i}\\
l^{i}%
\end{array}
\right]  =\left[
\begin{array}
[c]{c}%
0\\
0\\
0\\
r_{\Gamma}^{i}\\
0
\end{array}
\right]  ,\quad i=1,\dots,N_{S}.\label{eq:BDDC-substructure-problem}%
\end{equation}

\item Compute the \emph{coarse correction} $\eta_{C}^{i}$ by collecting the
coarse residual
\begin{equation}
r_{C}=\sum_{i=1}^{N_{S}}R_{C}^{iT}\Phi_{\Gamma}^{iT}r_{\Gamma}^{i},
\end{equation}
\ solving the global \emph{coarse problem}
\begin{equation}
S_{CC}\,\eta_{C}=r_{C},\label{eq:BDDC-coarse-problem}%
\end{equation}
and distributing the coarse correction%
\begin{equation}
\eta_{\Gamma C}^{i}=\Phi_{\Gamma}^{i}R_{C}^{i}\eta_{C},\quad i=1,\dots,N_{S}.
\end{equation}

\item Combine and average the corrections%
\begin{equation}
\lambda_{\Gamma}=-\sum_{i=1}^{N_{S}}R^{iT}W^{i}\left(  \eta_{\Gamma \Delta}^{i}%
+\eta_{\Gamma C}^{i}\right)  .\label{eq:BDDC-preconditioner}%
\end{equation}

\end{enumerate}
\end{algorithm}

We note that the factorisations of the matrices from
(\ref{eq:hybrid-fractured-augmented-substructure}) are also used for each
solution of (\ref{eq:BDDC-substructure-problem}).

In order to apply the existing BDDC theory for elliptic problems (e.g. \cite{Mandel-2007-BFM,Mandel-2005-ATP})
to the proposed preconditioner, 
we introduce some additional notation and make a few observations.
{Due to~(\ref{eq:substructure-schur}), the substructure corrections in~(\ref{eq:BDDC-substructure-problem})
can be written equivalently as}
\begin{equation}
\left[
\begin{array}
[c]{cc}%
-S^{i} & D^{iT}\\
D^{i} & 0
\end{array}
\right]  \left[
\begin{array}
[c]{c}%
\eta_{\Gamma \Delta}^{i}\\
l^{i}%
\end{array}
\right]  =\left[
\begin{array}
[c]{c}%
r_{\Gamma}^{i}\\
0
\end{array}
\right],
\label{eq:BDDC-substructure-problem-with-Schur}%
\end{equation}
and the construction of coarse basis functions~$\Phi_{\Gamma}^{i}$ in
(\ref{eq:hybrid-fractured-augmented-substructure}) as
\begin{equation}
\left[
\begin{array}
[c]{cc}%
-S^{i} & D^{iT}\\
D^{i} & 0
\end{array}
\right]  \left[
\begin{array}
[c]{c}
\Phi_{\Gamma}^{i}\\
L^{i}%
\end{array}
\right]  =\left[
\begin{array}
[c]{c}%
0\\
I
\end{array}
\right]. 
\label{eq:bddc-substructure-problem-with-schur-for-coarse-basis}%
\end{equation}

Next, let us formally write the operators and vectors in the block form
\begin{equation}
\label{eq:block_operators}
\lambda_{\Gamma}=\left[
\begin{array}
[c]{c}%
\lambda_{\Gamma}^{1}\\
\vdots\\
\lambda_{\Gamma}^{N_{S}}%
\end{array}
\right]  ,\ R=\left[
\begin{array}
[c]{c}%
R^{1}\\
\vdots\\
R^{N_{S}}%
\end{array}
\right]  ,\ W=\left[
\begin{array}
[c]{ccc}%
W^{1} &  & \\
& \ddots & \\
&  & W^{N_{S}}%
\end{array}
\right]  ,\ S=\left[
\begin{array}
[c]{ccc}%
S^{1} &  & \\
& \ddots & \\
&  & S^{N_{S}}%
\end{array}
\right]  .
\end{equation}

By grouping the steps of Algorithm~\ref{alg:bddc} and using (\ref{eq:BDDC-substructure-problem-with-Schur}), 
the operator of the BDDC preconditioner can be formally written as
\begin{equation}
\label{eq:Mbddc-formal}
M_{BDDC}           = R^{T} W \widetilde{S}^{-1} WR,
\end{equation}
where
\begin{equation}
\label{eq:Stilde}
\widetilde{S}^{-1} = 
-\left\{ 
\underset{\scriptsize i = 1,\dots,N_S}{\mbox{diag}}
\left(
\left[
\begin{array}
[c]{cc}%
I_{{\Lambda}_{\Gamma}^{i}} & 0 
\end{array}
\right]  
\left[
\begin{array}
[c]{cc}%
-S^{i} & D^{iT}\\
D^{i} & 0
\end{array}
\right]^{-1} 
\left[
\begin{array}
[c]{c}%
I_{{\Lambda}_{\Gamma}^{i}} \\ 
0 
\end{array}
\right]  
\right) 
 +
\left(
\sum_{i=1}^{N_{S}}R_{C}^{iT}\Phi_{\Gamma}^{iT} 
\right)^T
S_{CC}^{-1}
\left(
\sum_{i=1}^{N_{S}}R_{C}^{iT}\Phi_{\Gamma}^{iT}
\right)
\right\}.
\end{equation}
The first term in $\widetilde{S}^{-1}$ corresponds to substructure corrections and the second term to the coarse correction
(steps 2 and 3 of Algorithm~\ref{alg:bddc}), 
and $I_{{\Lambda}_{\Gamma}^{i}}$ is the identity matrix in ${\Lambda}_{\Gamma}^{i}$.
From (\ref{eq:Mbddc-formal}) and (\ref{eq:Stilde}), one can readily see that $M_{BDDC}$ is symmetric.

\begin{assumption}
\label{as:nullS}Let us assume that
\begin{equation}
\mathrm{null}\,S^{i} \perp \mathrm{null}\,D^{i}.
\end{equation}
\end{assumption}

In order to satisfy Assumption~\ref{as:nullS}, we must prescribe enough coarse
degrees of freedom as constraints along with Robin boundary conditions
(\ref{eq:newton-1})--(\ref{eq:newton-2}) and (\ref{eq:newton-3}) at each fracture within
substructure$~\Omega^{i}$. 
Since constrains in $D^{i}$ are linearly independent, $D^{iT}$ has full column rank. 
In particular,
Assumption~\ref{as:nullS} is satisfied when arithmetic averages are used on
each substructure face (and eventually edge) as constraints.

\begin{lemma}
\label{lemma:positive_definite_preconditioner} The operator $\widetilde{S}^{-1}$ in 
preconditioner $M_{BDDC}$ is symmetric and positive definite on the space $\widetilde{\Lambda}_{\Gamma}$.
\end{lemma}

\begin{proof} 

The space $\widetilde{\Lambda}_{\Gamma}$ is decomposed into the substructure spaces and the coarse space,
\begin{equation}
\widetilde{\Lambda}_{\Gamma} = \widetilde{\Lambda}_{\Gamma\Delta} \oplus \widetilde{\Lambda}_{\Gamma C}.
\end{equation}
To achieve this splitting, each local space ${\Lambda}_{\Gamma}^{i}$ is decomposed into subspaces
\begin{equation}
{\Lambda}_{\Gamma}^{i} = \mathrm{null}\,D^{i} \oplus \mathrm{range}\,\Phi_{\Gamma}^{i},
\label{eq:decomposition_of_subdomain_space}
\end{equation}
corresponding to the \emph{substructure space} $\widetilde{\Lambda}_{\Gamma\Delta}^{i}$, 
and the \emph{coarse space} on substructure $\Omega^{i}$, $\widetilde{\Lambda}_{\Gamma C}|_{\Omega^{i}}$, respectively.
To analyse this decomposition, 
{let us recall that $S^{i}$ is a positive semi-definite matrix and write
(\ref{eq:bddc-substructure-problem-with-schur-for-coarse-basis}) in detail as}
\begin{align}
\label{eq:bddc-substructure-problem-with-schur-for-coarse-basis-1}%
S^{i}\Phi_{\Gamma}^{i} & = D^{iT} L^{i}, \\
\label{eq:bddc-substructure-problem-with-schur-for-coarse-basis-2}%
D^{i} \Phi_{\Gamma}^{i}& = I.
\end{align}
From (\ref{eq:bddc-substructure-problem-with-schur-for-coarse-basis-1}),
\begin{equation}
\mathrm{range}\,(S^{i} \Phi_{\Gamma}^{i}) \subset \mathrm{range}\,D^{iT} \perp \mathrm{null}\,D^{i},
\label{eq:S-orthogonality}
\end{equation}
which in turn, similarly to \cite[Lemma~8]{Mandel-2005-ATP}, 
gives for any $\chi^{i}_{\Delta} \in \mathrm{null}\,D^{i}$ and $\xi^{i}_{C} \in \mathrm{range}\,\Phi_{\Gamma}^{i}$
\begin{equation}
\label{eq:S-orthogonality2}
\chi^{iT}_{\Delta} S^{i} \xi^{i}_{C} = 0.
\end{equation}
From (\ref{eq:bddc-substructure-problem-with-schur-for-coarse-basis-2}),
the matrix $\Phi_{\Gamma}^{i}$ has full column rank,
and 
\begin{equation}
\mathrm{null}\,D^{i} \cap \mathrm{range}\,\Phi_{\Gamma}^{i} = {0}.
\end{equation}
Finally, from Assumption~\ref{as:nullS} and (\ref{eq:decomposition_of_subdomain_space}), 
\begin{equation}
\mathrm{null}\,S^{i} \subset \mathrm{range}\,\Phi_{\Gamma}^{i}.
\end{equation}

Decomposition of the subdomain space (\ref{eq:decomposition_of_subdomain_space})
implies decomposition of a function $\zeta^{i} \in {\Lambda}_{\Gamma}^{i}$ to $\zeta^{i} =  \zeta^{i}_{\Delta} + \zeta^{i}_{C}$,
where $\zeta^{i}_{\Delta} \in \mathrm{null}\,D^{i}$, $\zeta^{i}_{C} \in \mathrm{range}\,\Phi_{\Gamma}^{i}$, 
and $\zeta^{iT}_{\Delta} S^{i} \zeta^{i}_{C} = 0$ by (\ref{eq:S-orthogonality2}).

Let us first analyse the substructure corrections. 
Following \cite[Section~3.3]{Benzi-2005-NSS}, the matrix of 
(\ref{eq:BDDC-substructure-problem-with-Schur})
is invertible due to Assumption~\ref{as:nullS}.
If we define, in addition, a matrix $Q^{i}$ with orthonormal columns forming a basis of $\mathrm{null}\,D^{i}$, i.e. 
\begin{equation}
\mathrm{range}\,Q^{i} = \mathrm{null}\,D^{i}, \qquad Q^{iT}Q^{i} = I,
\end{equation}
we have
\begin{equation}
\left[
\begin{array}
[c]{cc}%
I_{{\Lambda}_{\Gamma}^{i}} & 0 
\end{array}
\right]  
\left[
\begin{array}
[c]{cc}%
-S^{i} & D^{iT}\\
D^{i} & 0
\end{array}
\right]^{-1} 
\left[
\begin{array}
[c]{c}%
I_{{\Lambda}_{\Gamma}^{i}} \\ 
0 
\end{array}
\right]  
= 
-Q^{i}\left(Q^{iT} S^{i} Q^{i}\right)^{-1}Q^{iT}.
\end{equation}
The matrix $\left(Q^{iT} S^{i} Q^{i}\right)^{-1}$ is symmetric positive definite,
and consequently, for any $\zeta^{i} \in {\Lambda}_{\Gamma}^{i}$,
\begin{equation}
-\zeta^{iT}Q^{i}\left(Q^{iT} S^{i} Q^{i}\right)^{-1}Q^{iT}\zeta^{i} \le 0
\end{equation}
with equality if $\zeta^{i} = \zeta^{i}_{C} \in \mathrm{range}\,\Phi_{\Gamma}^{i}$.

Next, let us turn towards the coarse correction.
Formula~(\ref{eq:scci_def}) for$~S_{CC}^{i}$ can be written equivalently as
\begin{equation}
S_{CC}^{i}=-\Phi_{\Gamma}^{iT}S^{i}\Phi_{\Gamma}^{i}.
\end{equation}
{Since the term on the right-hand side is just a (negative) Galerkin projection of 
the positive semi-definite matrix $S^{i}$, matrix $S_{CC}^{i}$ is symmetric negative semi-definite.
If at least one substructure is equipped with {natural} boundary conditions, the matrix $S_{CC}$ assembled by
(\ref{eq:Scc_def}) becomes symmetric negative definite, and so is $S_{CC}^{-1}$.}

We have just verified the negative definiteness of the principal parts of $\widetilde{S}^{-1}$,
and the desired positive definiteness is obtained through the change of sign in front of the braces
in~(\ref{eq:Stilde}).
\end{proof}

In view of Lemma~\ref{lemma:positive_definite_preconditioner},
the standard condition number bound follows
from \cite[Lemma~2]{Mandel-2007-BFM}. 

\begin{theorem}
\label{th:bddc_bound} The condition number~$\kappa$ of the preconditioned
operator~$M_{BDDC}\widehat{S}$ satisfies
\begin{equation}
\kappa\leq\omega=\max_{\lambda_{\Gamma}\in\widetilde{\Lambda}_{\Gamma}}%
\frac{\left\Vert RR^{T}W\lambda_{\Gamma}\right\Vert _{S}^{2}}{\left\Vert
\lambda_{\Gamma}\right\Vert _{S}^{2}}.\label{eq:cond-bound-algebraic}%
\end{equation}

\end{theorem}

{The norms in (\ref{eq:cond-bound-algebraic}) are induced by 
the matrix $S$ defined in (\ref{eq:block_operators}) for all functions 
$\lambda_{\Gamma}\in\widetilde{\Lambda}_{\Gamma}$.}

In addition, in the case of a single mesh dimension in either 2D or 3D, 
{and under the assumption of substructure-wise constant hydraulic conductivities,}
it has been also derived in~\cite[Lemma~5.5 and Theorem~6.1]{Tu-2007-BAF} that the
condition number bound$~\omega$ satisfies
\begin{equation}
\omega\leq C\left(  1+\log\frac{H}{h}\right)  ^{2}%
,\label{eq:cond-bound-analytic}%
\end{equation}
where $H$ is the characteristic size of geometric substructures. 
{
We note that
the bound~(\ref{eq:cond-bound-analytic}) implies that, for a fixed relative subdomain size~$H/h$,
the condition number is independent of the problem size.

It is worth emphasising that Theorem~\ref{th:bddc_bound} is also valid for combined mesh dimensions.
However, several simplifications are employed in \cite{Tu-2007-BAF} to obtain 
(\ref{eq:cond-bound-analytic}), which are not satisfied in the set-up considered in this paper.
In particular 
(i) hydraulic conductivity coefficient here is, in general, not substructure-wise constant nor isotropic, 
{(ii) it is not clear whether, in presence of fractures, the interpolation operator onto 
a~conforming mesh introduced in \cite{Tu-2007-BAF}
can be constructed and 
bounded in the $H^1$ norm.}
}

\section{Scaling weights in BDDC}
\label{sec:scaling}

Let us now discuss the choice of entries in the diagonal weight matrices $W^{i}%
$. These matrices play an
important role in the BDDC method, both in the theory,
cf.~Theorem~\ref{th:bddc_bound} or
\cite{Tu-2007-BAF,Mandel-2007-BFM,Mandel-2003-CBD}, and in the
computations, cf.~\cite{Klawonn-2008-AFA,Certikova-2013-SIW}. Three possible
choices are also studied numerically in Section~\ref{sec:results_engineering}.
The basic choice is presented by the \emph{arithmetic average} taken from
values at the neighbouring substructures. In this simplest construction, the
entry corresponding to Lagrange multiplier $\lambda_{\Gamma,j}^{i}$ is given
by the inverse counting function as
\begin{equation}
W_{jj}^{i}=\frac{1}{card(\mathcal{I}_{j})}, \label{eq:scaling_arithmetic}%
\end{equation}
where $card(\mathcal{I}_{j})$ is the number of substructures in the set
$\mathcal{I}_{j}$ of indices of substructures to which $\lambda_{\Gamma,j}%
^{i}$ belongs. For 2D or 3D meshes without fractures, $W_{jj}^{i}=\nicefrac{1}{2}$
for Raviart-Thomas elements. However, since several two-dimensional fractures
can meet in our setting, smaller weights can {occasionally}
appear at such regions.

While arithmetic average is sufficient for problems with homogeneous
coefficients, it is well known that for problems with large variations in
material properties along the interface, it is necessary to incorporate their
values into the (weighted) average to obtain a robust method. This gives rise to
the~$\rho$-scaling, for which
\begin{equation}
\label{eq:scaling_rho}W^{i}_{jj} = \frac{\rho_{i}}{\sum_{k = 1}%
^{card(\mathcal{I}_{j})} \rho_{k}},
\end{equation}
where $\rho_{k}$ is a material characteristic for substructure~$\Omega^{k}$.
This choice is robust with respect to jumps in coefficients across the interface,
cf.~\cite{Tu-2007-BAF, Toselli-2005-DDM}; however, coefficients are assumed
constant for each substructure. This requirement is very restrictive for
practical computations with quickly varying coefficients, and we employ a
generalisation which takes into account the material coefficient of the
element to which the Lagrange multiplier $\lambda_{\Gamma,j}^{i}$ corresponds.
In our case, we use $\rho_{i} = d/\mbox{tr}\,(\Bbbk^{-1})$, where $d\in\{1,2,3\}$ is the
dimension of the element $T^{i}$. 
This value can be seen as a representative hydraulic conductivity on the element.

Finally, we propose a modification of the popular scaling by diagonal
stiffness~\cite{Klawonn-2008-AFA}. In the usual diagonal stiffness approach,
the optimal weight, which is the diagonal entry of the Schur complement, is
approximated by the diagonal entry of the original substructure matrix.
However, this is not directly applicable to the indefinite system
(\ref{eq:hybrid-fractured-substructure}), as, in general, matrix $C^{i}$
contains only seldom nonzeros on the diagonal. For this reason, we approximate
the diagonal of the Schur complement as
\begin{equation}
\label{eq:scaling_diagonal}W^{i}_{jj} = \widetilde{C}^{i}_{\Gamma\Gamma,jj} +
\frac{1}{A^{i}_{kk}},
\end{equation}
where the index $k$ corresponds to the row in block $A^{i}$ of the element
face to which the Lagrange multiplier~$\lambda_{\Gamma,j}^{i}$ belongs.

Using the diagonal stiffness scaling in connection with the standard Lagrange
finite elements may lead to poor convergence for problems with rough interface
\cite{Klawonn-2008-AFA,Certikova-2013-SIW}, for which the diagonal stiffness
can vary quickly even for smooth problems with constant coefficients 
{on uniform meshes}. 
This is a severe issue for practical computations, in which graph partitioners are
typically used for creating substructures. 
However, this issue is 
{not as pronounced}
for Raviart-Thomas elements, for which  
only one element contributes to
the stiffness on the diagonal at an interface degree of freedom,
{and thus irregularities caused by changing number of elements contributing to an interface weight 
cannot occur.}
On the other
hand, an advantage of the diagonal stiffness scaling is the fact,
that---unlike the $\rho$-scaling---it takes into account the shape and
relative sizes of elements, which vary considerably in engineering
applications, as well as the effect of $\delta_{d}$ introduced in~(\ref{eq:weak_term_a}) 
and~(\ref{eq:weak_term_f}). Unless stated otherwise,
scaling~(\ref{eq:scaling_diagonal}) is used in the computations presented in
Section~\ref{sec:numerical}.

\section{The parallel solver}

\label{sec:solver}

The basis for an efficient parallel implementation of the method described in
previous sections was obtained by combining two existing open-source software
packages: the finite element package \textsl{Flow123d}%
\footnote{%
\url{http://flow123d.github.io}} (version 1.6.5) for
underground fluid flow simulations and the BDDC-based solver \textsl{BDDCML}%
\footnote{\url{http://users.math.cas.cz/~sistek/software/bddcml.html}} (version
2.0) used for the solution of the resulting system of equations. However, minor
changes have been made to both codes to support the specific features, such as
the weights (\ref{eq:scaling_rho}) and~(\ref{eq:scaling_diagonal}).

The \textsl{Flow123d} package has been developed for modelling complex
behaviour of underground water flow and pollution transport. However, only the
simple flow in a fully saturated porous media described by Darcy's law is
considered in this paper. To accurately account for fractures in the medium,
such as granite rock, the solver allows us to combine finite elements of
different dimensions: the three-dimensional elements of porous media are
combined with two-dimensional elements modelling planar fractures, which may be in
turn connected in one-dimensional elements for channels. Raviart-Thomas
elements are consistently used throughout such discretisation. Although the
fractures are also modelled as porous media, their hydraulic conductivity is by
orders of magnitude higher than that of the main porous material of the
domain. In addition, the finite element discretisations are typically not
uniform within the domain, and the relative sizes of elements may also vary by
orders of magnitude. Both these aspects give rise to very poorly conditioned
linear systems, which are very challenging for iterative solvers.
The \textsl{Flow123d} solver has been developed for over 10 years and it is
written in C/C++ programming language with object-oriented design and
parallelism through MPI.

The \textsl{BDDCML} is a library for solving algebraic systems of linear
equations by means of the BDDC method. The package supports the
Adaptive-Multilevel BDDC method \cite{Sousedik-2013-AMB} suitable for very
high number of substructures and computer cores, although we only use the
standard (non-adaptive two-level) BDDC method from
\cite{Tu-2007-BAF,Dohrmann-2003-PSC} for the purpose of this paper. The
\textsl{BDDCML} library is typically interfaced by finite element packages,
which may provide the division into substructures. This feature is used in our
current implementation, in which the division into non-overlapping
substructures is constructed within the \textsl{Flow123d} using the
\textsl{METIS} (version 5.0) package \cite{Karypis-1998-FHQ}. One substructure
is assigned to a processor core in the current set-up of the parallel solver,
although \textsl{BDDCML} is more flexible in this respect. The library
performs the selection of additional corners by the face-based algorithm from
\cite{Sistek-2012-FSC}. The \textsl{BDDCML} package is written in Fortran 95
and parallelised through MPI.

The \textsl{BDDCML} solver relies on a serial instance of the \textsl{MUMPS}
direct solver \cite{Amestoy-2000-MPD} for the solution of each local discrete
Dirichlet problem (\ref{eq:discrete-dirichlet-problem-substructure}) as well
as for the solution of each local discrete Neumann
problem~(\ref{eq:BDDC-substructure-problem}). The coarse
problem~(\ref{eq:BDDC-coarse-problem}) is solved by a parallel instance of
\textsl{MUMPS}. The main difference from using \textsl{BDDCML} for symmetric
positive definite problems is the need to use the $LDL^{T}$ factorisation of
general symmetric matrices for problems
(\ref{eq:discrete-dirichlet-problem-substructure}), which are saddle-point
(i.e. indefinite) systems in the present setting.

Although the original system (\ref{eq:hybrid-m-fractured}) is indefinite,
system (\ref{eq:S_hat-system}) is symmetric positive definite, which allows the
use of the preconditioned conjugate gradient (PCG) method. One step of BDDC is
used as the preconditioner within the PCG method applied to problem
(\ref{eq:S_hat-system}). The matrix of problem (\ref{eq:S_hat-system}) is not
explicitly constructed in the solver, and only its actions on vectors are
computed following (\ref{eq:discrete-dirichlet-problem-substructure})--(\ref{eq:schur-assembly}).

\begin{remark}
{In our implementation, we change the sign neither in 
the action of $S^i$ (\ref{eq:substructure-schur-multiple}) nor in the action of  
the preconditioner $M_{BDDC}$ (\ref{eq:BDDC-preconditioner}).}
Since both are then strictly negative definite, the product $M_{BDDC}\widehat{S}$ is 
the same as if both signs were changed, and the PCG method runs correctly.
In this way, no changes are necessary 
in an implementation developed for symmetric positive definite problems.
\end{remark}

\section{Numerical results}

\label{sec:numerical}

In this section, we investigate the performance of the algorithm and its
parallel implementation on several benchmark problems in 2D and 3D, and on two
geoengineering problems of existing localities in 3D. 
{For the two benchmark problems without fractures 
we perform weak scaling tests. 
For the benchmark problem with fractures
and for the geoengineering problems, 
we perform strong scaling tests with the problem
size fixed and increasing number of processor cores.}
In all cases, the PCG
iterations are run until the relative norm of residual $\|r^{(k)}%
\|/\|\widehat{b}\| < 10^{-7}$. If not stated otherwise, the proposed scaling
by diagonal stiffness (\ref{eq:scaling_diagonal}) is used within the averaging
operator of BDDC.

\subsection{Results for benchmark problems}

\label{sec:results_benchmark}

First, the performance of the solver is investigated on a unit square and a
unit cube discretised solely using two-dimensional and three-dimensional
finite elements, respectively. For this reason, block $\ol{C}$ in
system~(\ref{eq:hybrid-m-fractured}), which is related to combining elements
of different dimension, is zero, and the problem reduces to the standard
problem~(\ref{eq:hybrid-m}). The sequence of unstructured meshes is
approximately uniform for both problems, and the problems do not contain any
jumps in material coefficients. 
In Figs.~\ref{fig:square2d_problem} and \ref{fig:cube3d_problem}, example meshes
and the resulting pressure head and velocity fields are presented. While
gravity is present in the 3D case, its effect is not considered in the 2D case.

The results of the weak scaling tests are summarised in
Tables~\ref{tab:2d_square_weak_scaling} and \ref{tab:3d_cube_100k_weak_scaling}.
{To give a better view, the resulting solution times for different problem sizes are also
visualised in Fig.~\ref{fig:timing_100k}.}
In these tables, $N$ denotes the
number of substructures and processors, $n$ is the size of the global problem
(\ref{eq:hybrid-m-fractured}), $n_{\Gamma}$ is the size of the interface
problem (\ref{eq:S_hat-system}), $n_{f}$ denotes number of faces, $n_{c}$
denotes number of corners, `its.' stands for resulting number of PCG
iterations, and `cond.' is the approximate condition number computed from the
Lanczos sequence in PCG. We report separately the time spent in preconditioner
set-up, the time spent by PCG iterations, and the total time for the whole solve.

In these weak scaling tests, the number of unknowns per core is kept
approximately constant around $10^{5}$. 
These weak scaling tests were performed using up to 
64 cores of the {SGI Altix UV} supercomputer at the Supercomputing Centre of the Czech Technical University
in Prague.
{The computer contains twelve Intel Xeon processors, each with six cores at frequency 2.67 GHz.
Intel compilers version 12.0 were used.}

The numbers of PCG iterations and condition number estimates in
Tables~\ref{tab:2d_square_weak_scaling} and
\ref{tab:3d_cube_100k_weak_scaling} confirm the expected numerical scalability
of the BDDC method, which is well known for symmetric positive definite
problems as well as for Darcy's flow problems
\cite{Tu-2007-BAF,Mandel-2003-CBD}.
{The slight irregularities in the condition number in Table~{\ref{tab:3d_cube_100k_weak_scaling}}
are probably caused by using non-nested unstructured meshes.}

Looking at times in these tables and in Fig.~\ref{fig:timing_100k}, we can see
almost optimal scaling, with only mild growth of times with number of cores.
The numbers of PCG iterations are higher in 3D, and the time spent in PCG
iterations grows proportionally, while the time spent in the set-up phase does
not differ considerably between two-dimensional and three-dimensional setting
and dominates the overall time.

\begin{figure}[ptbh]
\includegraphics[width=0.47\textwidth]{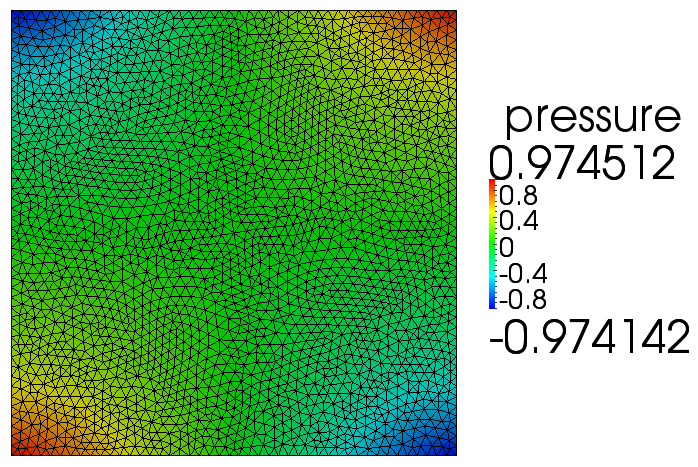} 
\includegraphics[width=0.50\textwidth]{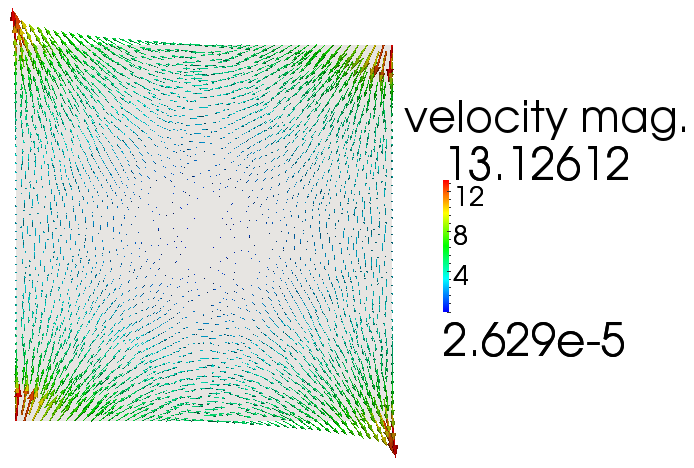}
\caption{\label{fig:square2d_problem}
Example of solution to the model square problem containing only 2D
elements, plot of pressure head with mesh (left) and velocity vectors (right).}
\end{figure}

\begin{figure}[ptbh]
\includegraphics[width=0.47\textwidth]{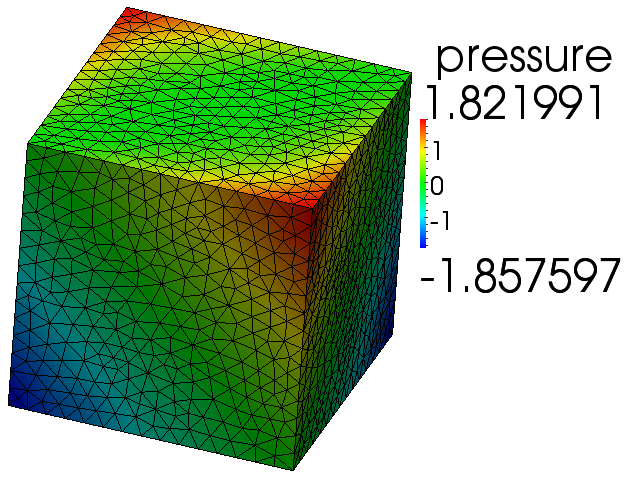} 
\includegraphics[width=0.50\textwidth]{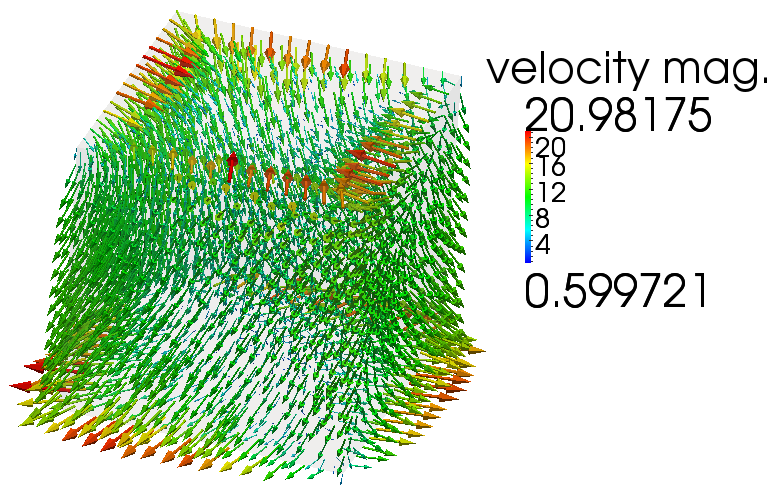}
\caption{\label{fig:cube3d_problem}
Example of solution to the model cube problem containing only 3D
elements, plot of pressure head with mesh (left) and velocity vectors (right).}
\end{figure}

\begin{table}[ptbh]
\begin{center}
\begin{tabular}
[c]{|ccc|ccc|cc|ccc|}\hline
\multirow{2}{*}{$N$} & \multirow{2}{*}{$n$} & \multirow{2}{*}{$n/N$} &
\multirow{2}{*}{$n_{\Gamma}$} & \multirow{2}{*}{$n_f$} &
\multirow{2}{*}{$n_c$} & \multirow{2}{*}{its.} & \multirow{2}{*}{cond.} &
\multicolumn{3}{c|}{time (sec)}\\
&  &  &  &  &  &  &  & set-up & PCG & solve\\\hline
2 & 207k & 103k & 155 & 1 & 2 & 7 & 1.37 & 8.3 & 1.6 & 9.9\\
4 & 440k & 110k & 491 & 5 & 10 & 8 & 1.60 & 12.2 & 2.2 & 14.4\\
8 & 822k & 103k & 1.2k & 13 & 26 & 9 & 1.78 & 11.0 & 2.5 & 13.5\\
16 & 1.8M & 111k & 2.8k & 33 & 66 & 8 & 1.79 & 14.3 & 2.7 & 17.0\\
32 & 3.3M & 104k & 5.9k & 74 & 148 & 9 & 1.79 & 12.1 & 3.3 & 15.4\\
64 & 7.2M & 113k & 13.0k & 166 & 332 & 9 & 1.85 & 14.8 & 4.4 & 19.2\\\hline
\end{tabular}
\end{center}
\caption{\label{tab:2d_square_weak_scaling}
Weak scaling test for the 2D square problem, each substructure
problem contains approx. 100k unknowns.}
\end{table}

\begin{table}[ptbh]
\begin{center}
\begin{tabular}
[c]{|ccc|ccc|cc|ccc|}\hline
\multirow{2}{*}{$N$} & \multirow{2}{*}{$n$} & \multirow{2}{*}{$n/N$} &
\multirow{2}{*}{$n_{\Gamma}$} & \multirow{2}{*}{$n_f$} &
\multirow{2}{*}{$n_c$} & \multirow{2}{*}{its.} & \multirow{2}{*}{cond.} &
\multicolumn{3}{c|}{time (sec)}\\
&  &  &  &  &  &  &  & set-up & PCG & solve\\\hline
2 & 217k & 108k & 884 & 1 & 3 & 11 & 2.88 & 11.7 & 2.3 & 14.0\\
4 & 437k & 109k & 2.3k & 6 & 18 & 12 & 3.04 & 11.7 & 2.5 & 14.2\\
8 & 945k & 118k & 5.7k & 21 & 63 & 15 & 12.00 & 15.4 & 4.0 & 19.3\\
16 & 1.6M & 103k & 12.8k & 56 & 168 & 16 & 6.58 & 12.9 & 4.0 & 17.0\\
32 & 3.4M & 106k & 29.8k & 132 & 401 & 18 & 10.10 & 15.4 & 5.2 & 20.6\\
64 & 6.1M & 95k & 59.6k & 307 & 931 & 19 & 16.58 & 13.7 & 6.3 & 20.0\\\hline
\end{tabular}
\end{center}
\caption{\label{tab:3d_cube_100k_weak_scaling}
Weak scaling test for the 3D cube problem, each substructure problem
contains approx. 100k unknowns.}
\end{table}

\begin{figure}[ptbh]
\begin{center}
\includegraphics[width=0.48\textwidth]{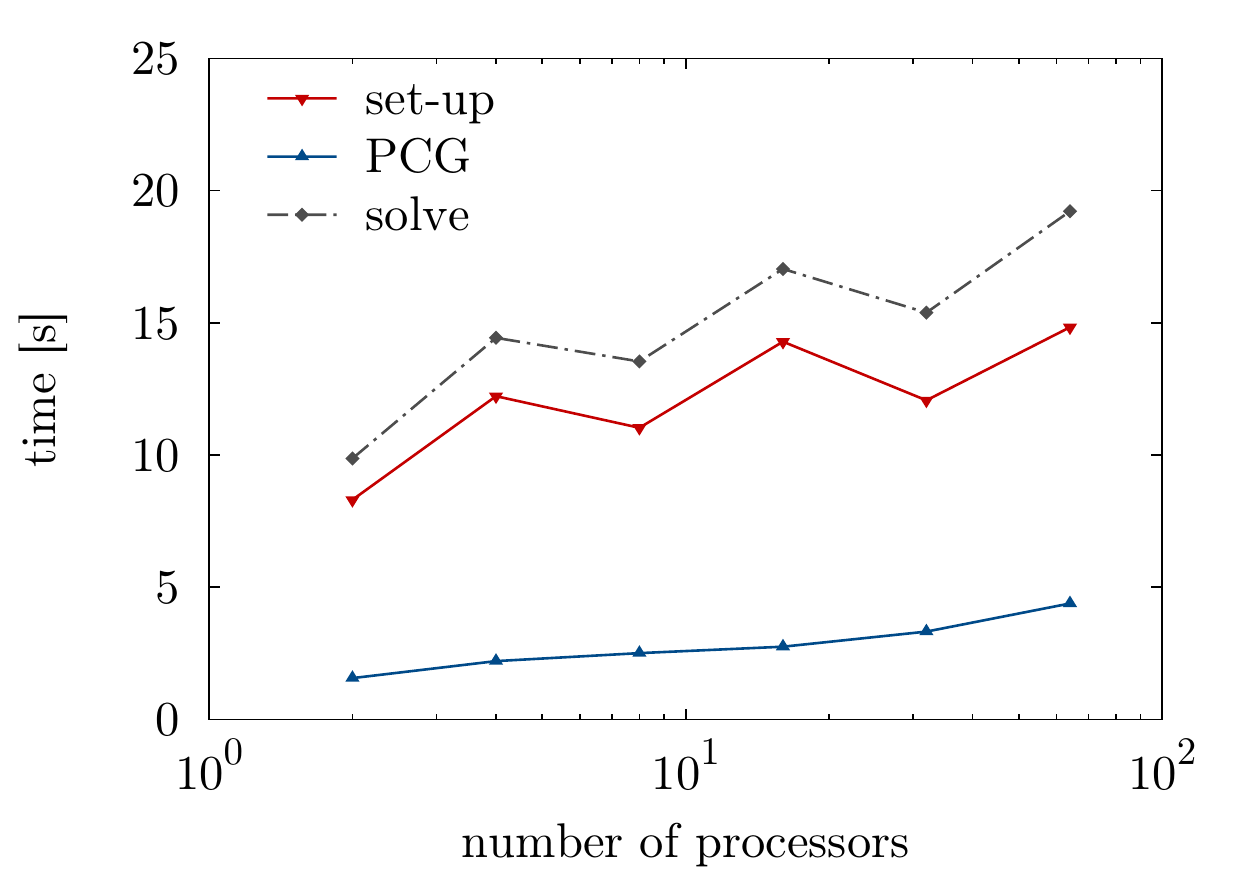} 
\includegraphics[width=0.48\textwidth]{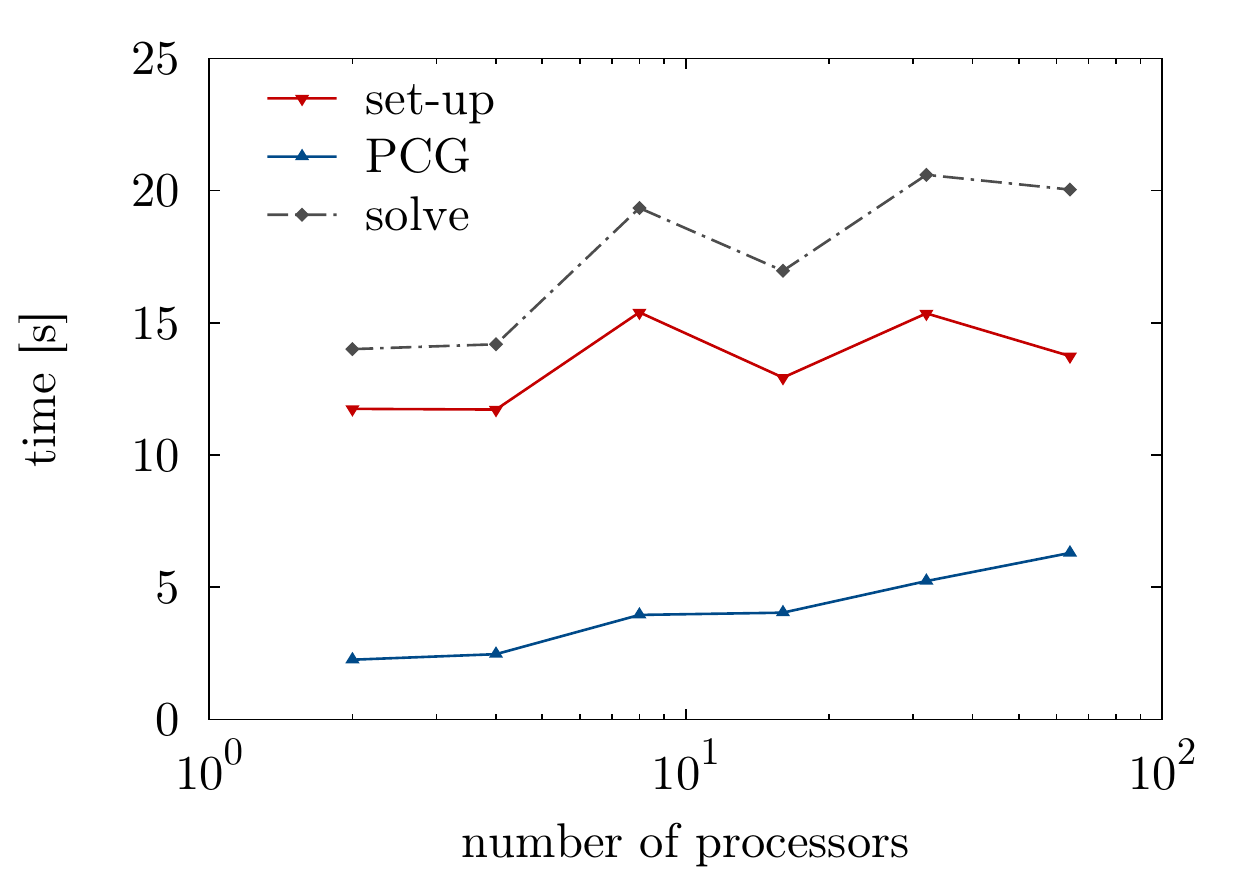}
\end{center}
\caption{\label{fig:timing_100k}
Weak scaling test for the 2D square problem (left), and the 3D cube
problem (right), approx. 100k unknowns per core.
Computational times separately
for set-up and PCG phases, and their sum (solve).}
\end{figure}

The next benchmark problem is considerably more complicated. It consists again
of a unit cube, which now contains four planar fractures aligned with
diagonals of a 2D cross-section. These four planar fractures meet at a 1D
channel in the centre of the cross-section. Therefore, the problem contains
the full possible combination of 3D, 2D and 1D finite elements. The tensor
$\Bbbk$ is isotropic, thus it is just a scalar multiple of identity. 
The corresponding scalar value is set to 10, 1, and 0.1 for 1D, 2D, and 3D elements, respectively.

We perform a strong scaling test with this problem, keeping the mesh size
fixed with approximately 2.1 million elements and 14.6 million degrees of
freedom. In Fig.~\ref{fig:cube123d_problem}, the computational mesh and the
resulting pressure head and velocity fields are presented. This scaling test
was computed on the Cray XE6 supercomputer \emph{Hector} at the Edinburgh
Parallel Computing Centre.
{This supercomputer is composed of 2816 nodes, 
each containing two AMD Opteron Interlagos processors with 16 cores at 2.3 GHz.
GNU compilers version 4.6 were employed.}

Results of the strong scaling test are summarised in
Table~\ref{tab:123d_cube_strong_scaling}, and the computing times are
visualised also in Fig.~\ref{fig:timing_strong_123d} together with the
parallel speed-up. The reference value for computing speed-up is the time on
16 cores, and the speed-up on $np$ processors is computed as
\begin{equation}
s_{np} = \frac{16 \ t_{16}}{t_{np}},
\end{equation}
where $t_{np}$ is the time on $np$ processors.

We can see that the number of PCG iterations grows with the number of
substructures for this problem which is also confirmed by the growing
condition number estimate. While the time spent in set-up phase scales very
well, the time spent in PCG grows together with the number of iterations. The
reason for this growth seem to be related to the larger interface, at which
more numerical difficulties appear. This seems to be related to more 1D-2D and
2D-3D connections at the interface and makes this difficult problem a good candidate for
using the \emph{Adaptive BDDC} method \cite{Sousedik-2013-AMB,Mandel-2012-ABT}. 
However, this will be the subject of a separate study.

\begin{figure}[ptbh]
\begin{center}
\includegraphics[width=0.47\textwidth]{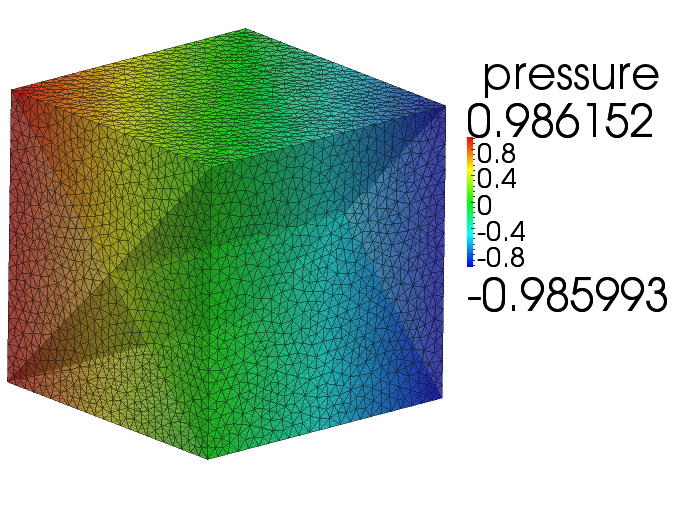}
\includegraphics[width=0.50\textwidth]{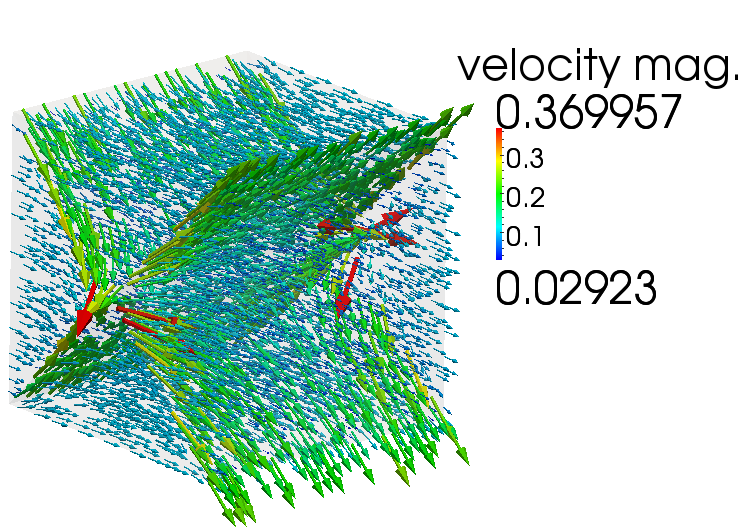}
\end{center}
\caption{\label{fig:cube123d_problem}
Example of solution to the model cube problem containing 1D, 2D, and
3D elements, plot of pressure head with mesh and fractures (left) and velocity
vectors (right).}
\end{figure}

\begin{table}[ptbh]
\begin{center}
\begin{tabular}
[c]{|cc|ccc|cc|ccc|}\hline
\multirow{2}{*}{$N$} & \multirow{2}{*}{$n/N$} &
\multirow{2}{*}{$n_{\Gamma}$} & \multirow{2}{*}{$n_f$} &
\multirow{2}{*}{$n_c$} & \multirow{2}{*}{its.} & \multirow{2}{*}{cond.} &
\multicolumn{3}{c|}{time (sec)}\\
&  &  &  &  &  &  & set-up & PCG & solve\\\hline
16 & 912k & 47k & 53 & 159 & 26 & 59.3 & 171.6 & 84.5 & 256.2\\
32 & 456k & 65k & 126 & 380 & 48 & 2091.0 & 90.1 & 109.8 & 200.0\\
64 & 228k & 86k & 301 & 914 & 81 & 1436.1 & 36.8 & 77.1 & 114.0\\
128 & 114k & 116k & 689 & 2076 & 109 & 2635.8 & 14.3 & 43.1 & 57.4\\
256 & 57k & 151k & 1436 & 4365 & 164 & 1700.5 & 6.7 & 31.2 & 38.0\\
512 & 28k & 196k & 3021 & 9244 & 254 & 42614.5 & 4.0 & 26.9 & 30.9\\\hline
\end{tabular}
\end{center}
\caption{\label{tab:123d_cube_strong_scaling}
Strong scaling test for the cube problem with 1D, 2D, and 3D
elements, size of the global problem is $n$ = 14.6M unknowns.}
\end{table}

\begin{figure}[ptbh]
\begin{center}
\includegraphics[width=0.48\textwidth]{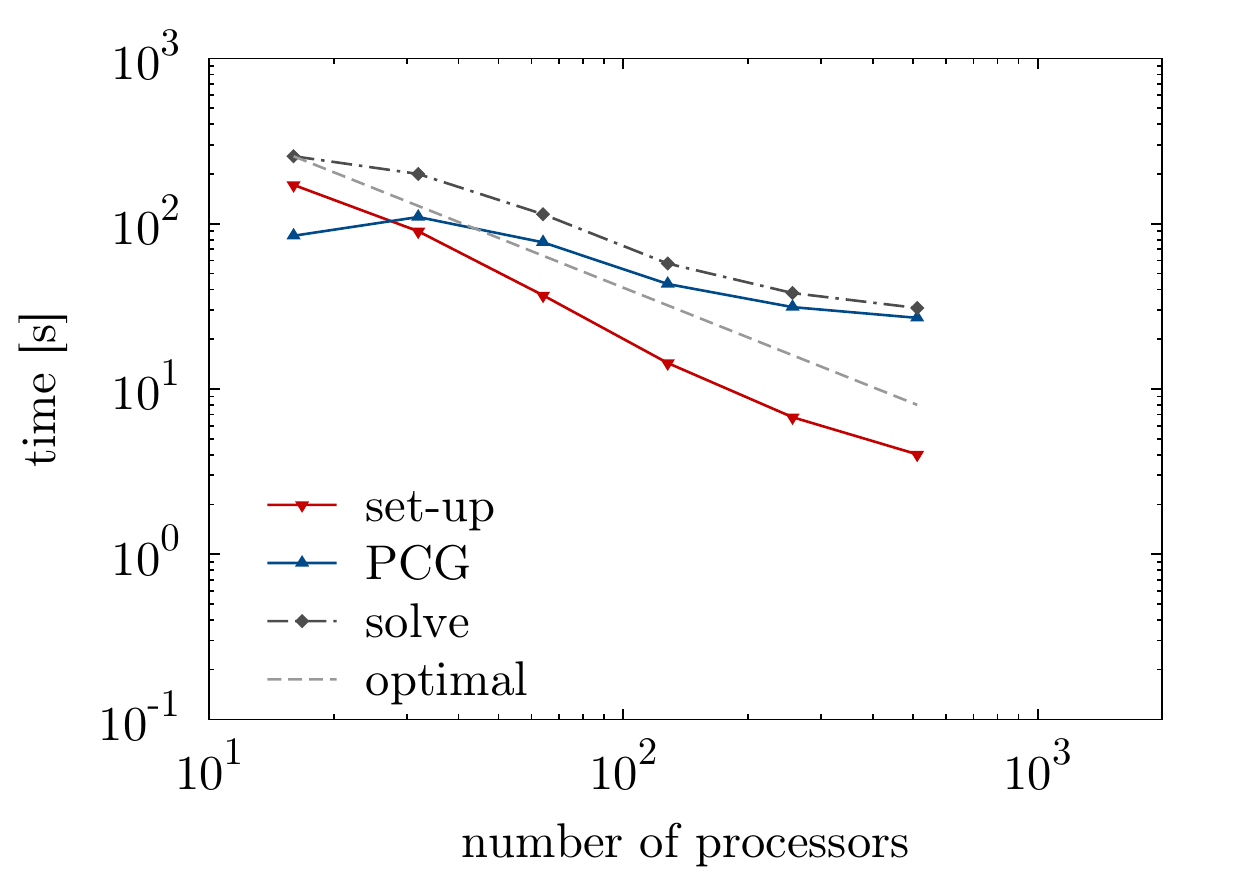} 
\includegraphics[width=0.48\textwidth]{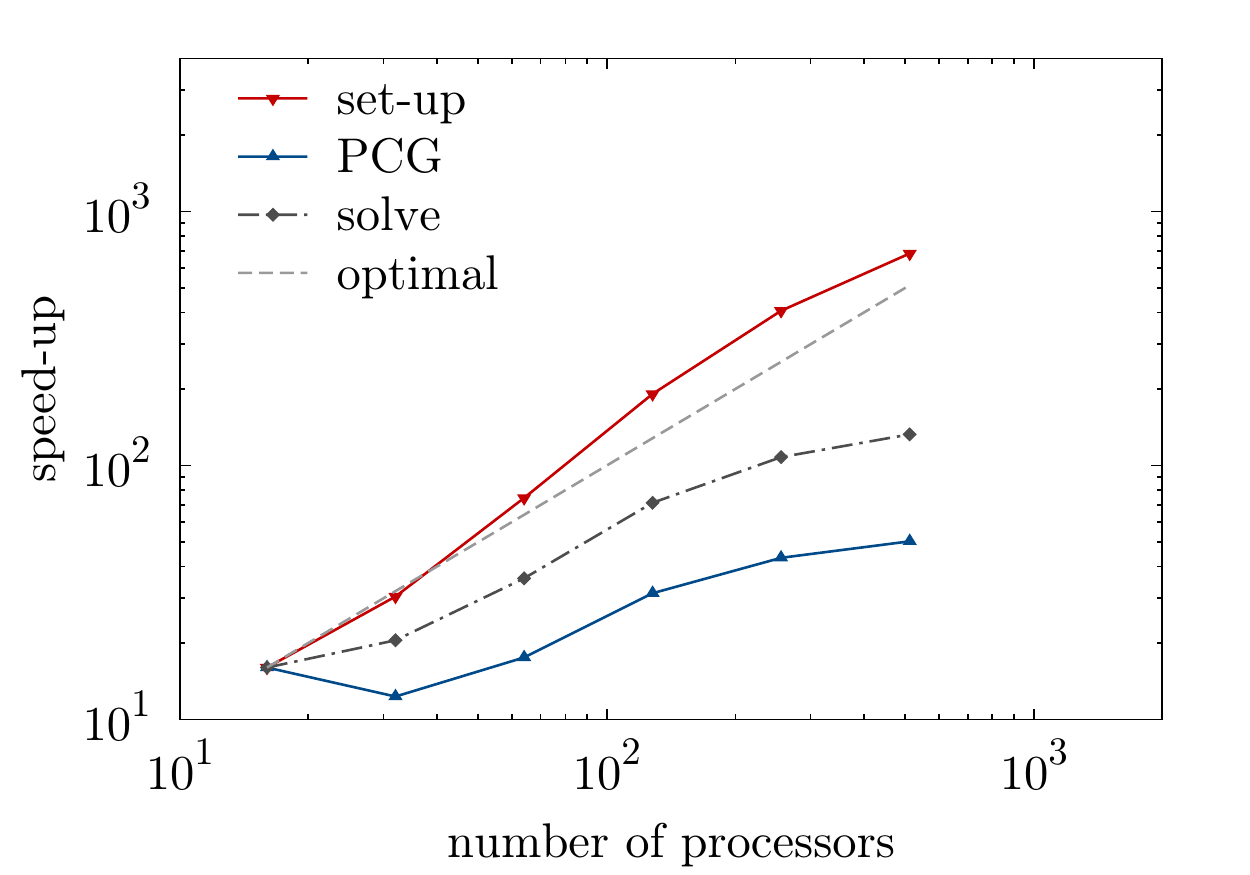}
\end{center}
\caption{\label{fig:timing_strong_123d}
Strong scaling test for the cube problem with 1D, 2D, and 3D elements
and 14.6M unknowns, computational time (left) and speed-up (right) separately
for set-up and PCG phases, and their sum (solve).}
\end{figure}

\subsection{Results for geoengineering problems}

\label{sec:results_engineering}

The performance of the algorithm and its parallel implementation has been
investigated on two engineering problems of underground flows within real
geologic locations. For both problems, the porous medium is fractured
granite rock, with the fractures modelled by two-dimensional elements.

The first problem is the \emph{Melechov locality}, which models one of the
candidate sites for a nuclear waste deposit to be build within the Czech
Republic in future. The goal is to model the underground flow and estimate the
speed at which an eventual radioactive pollution would spread. The
computational mesh contains 2.1 million finite elements resulting in 15
million unknowns. The geometry of the problem with the resulting distribution
of piezometric head and the finite element mesh is presented in
Fig.~\ref{fig:melechov_1}. The problem contains vertical two-dimensional
fractures visualised in Fig.~\ref{fig:melechov_2}. The maximal hydraulic
conductivity within the fractures is 6.3$\cdot$10$^{4}$ ms$^{-1}$, while the minimal
conductivity of the outer material is 6.0$\cdot$10$^{-3}$ ms$^{-1}$, the
transition coefficient $\sigma_{3}$ = 1 s$^{-1}$, and the effective thickness of
fractures $\delta_{2}$ = 0.1 m.

We perform a strong scaling test for this problem, keeping the problem size
fixed and increasing the number of substructures and computing cores. An
example of division into 64 substructures is presented in
Fig.~\ref{fig:melechov_2}. 
The scaling test was computed on the \emph{Hector} supercomputer.

Table~\ref{tab:melechov_large_strong_scaling} summarises the results of this
test. We can still see some growth of the number of iterations with the number
of substructures, which is however much milder than the growth observed for
the unit cube with fractures in Table~\ref{tab:123d_cube_strong_scaling}.
Correspondingly, the times reported in
Table~\ref{tab:melechov_large_strong_scaling} and visualised in
Fig.~\ref{fig:timing_Melechov_large} show an optimal scaling of the solver
over a large range of core counts.

\begin{figure}[ptbh]
\begin{center}
\begin{overpic}[width=0.7\textwidth]{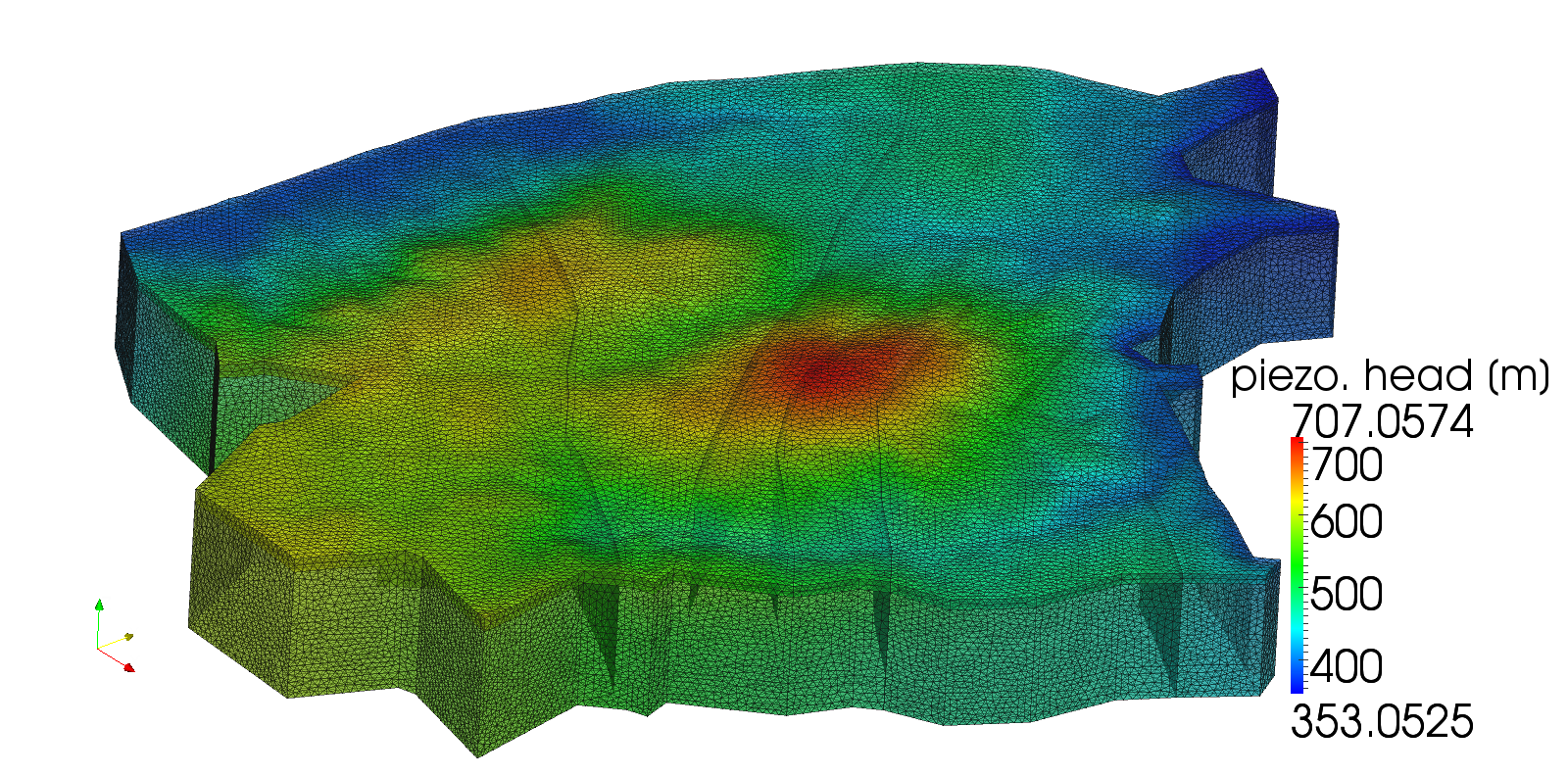}
\put(9,5){$x$}
\put(9,9){$y$}
\put(6,11.5){$z$}
\put(-10,18){$\Delta x=$9940 m}
\put(25,45){$\Delta y=$12100 m}
\put(87,35){$\Delta z=$1410 m}
\end{overpic}
\end{center}
\caption{\label{fig:melechov_1}
The problem of the \emph{Melechov locality} containing 2D and 3D
elements, mesh contains 2.1M elements and 15M unknowns. Plot of the
piezometric head. Data by courtesy of Ji{\v r}ina Kr{\' a}lovcov{\' a}.}
\end{figure}

\begin{figure}[ptbh]
\begin{center}
\includegraphics[width=0.4\textwidth]{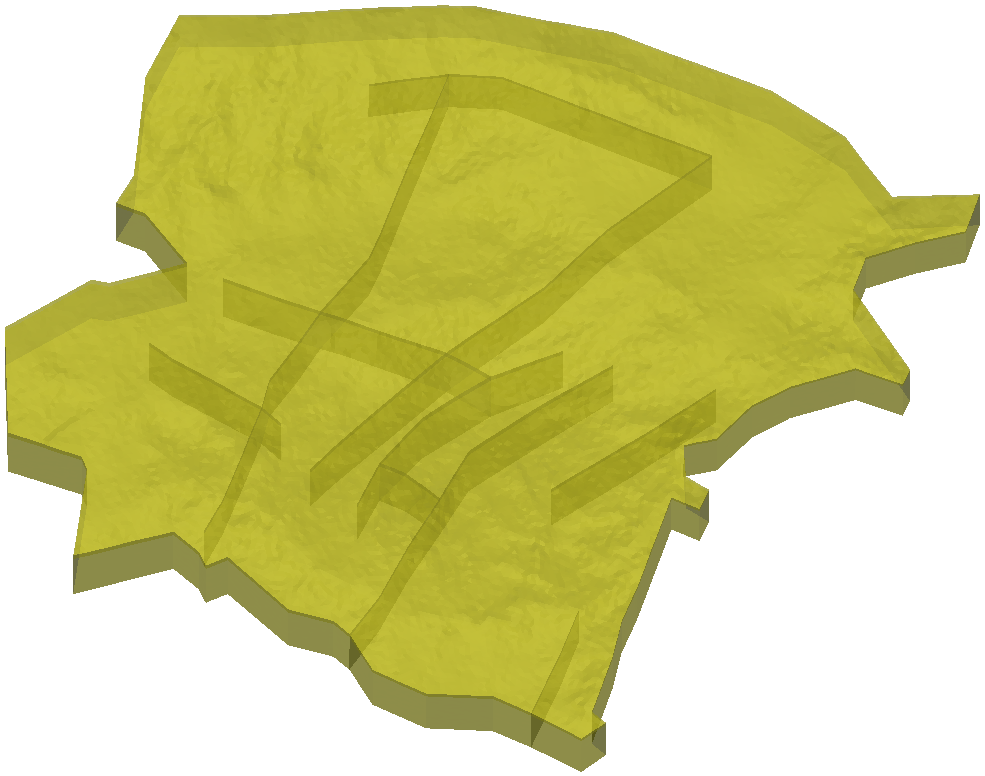} 
\includegraphics[width=0.4\textwidth]{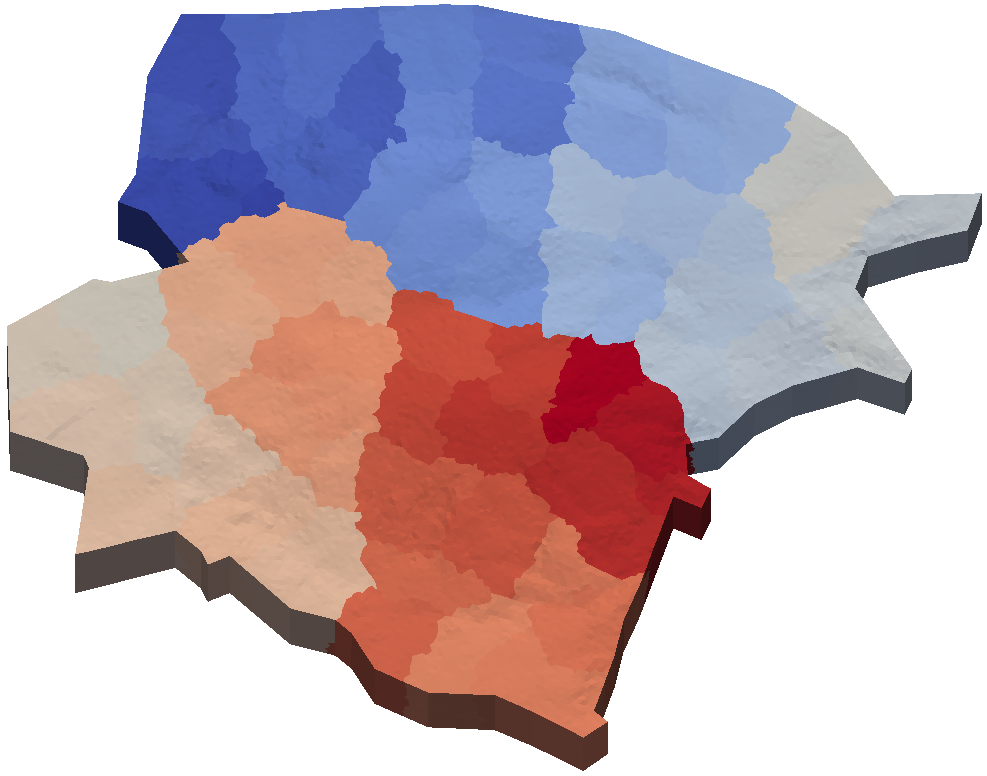}
\end{center}
\caption{\label{fig:melechov_2}The problem of the \emph{Melechov locality}; the system of fractures
(left) and an example division into 64 substructures (right).}
\end{figure}

\begin{table}[ptbh]
\begin{center}
\begin{tabular}
[c]{|cc|ccc|cc|ccc|}\hline
\multirow{2}{*}{$N$} & \multirow{2}{*}{$n/N$} &
\multirow{2}{*}{$n_{\Gamma}$} & \multirow{2}{*}{$n_f$} &
\multirow{2}{*}{$n_c$} & \multirow{2}{*}{its.} & \multirow{2}{*}{cond.} &
\multicolumn{3}{c|}{time (sec)}\\
&  &  &  &  &  &  & set-up & PCG & solve\\\hline
16 & 934k & 36k & 32 & 96 & 40 & 53.0 & 131.4 & 144.1 & 275.6\\
32 & 467k & 54k & 76 & 228 & 70 & 878.3 & 47.5 & 112.9 & 160.4\\
64 & 233k & 82k & 186 & 561 & 67 & 202.4 & 17.4 & 50.2 & 67.7\\
128 & 117k & 116k & 528 & 1592 & 69 & 237.6 & 7.9 & 23.1 & 31.1\\
256 & 58k & 155k & 1235 & 3747 & 96 & 5577.0 & 4.0 & 14.7 & 18.8\\
512 & 29k & 207k & 2699 & 8256 & 106 & 1658.1 & 2.2 & 8.3 & 10.5\\
1024 & 15k & 271k & 5711 & 17581 & 119 & 11554.5 & 2.1 & 7.0 & 9.2\\\hline
\end{tabular}
\end{center}
\caption{\label{tab:melechov_large_strong_scaling}
Strong scaling test for the problem of the \emph{Melechov locality}
containing 2D and 3D elements, size of the global problem is $n$ = 15M
unknowns.}
\end{table}

\begin{figure}[ptbh]
\begin{center}
\includegraphics[width=0.48\textwidth]{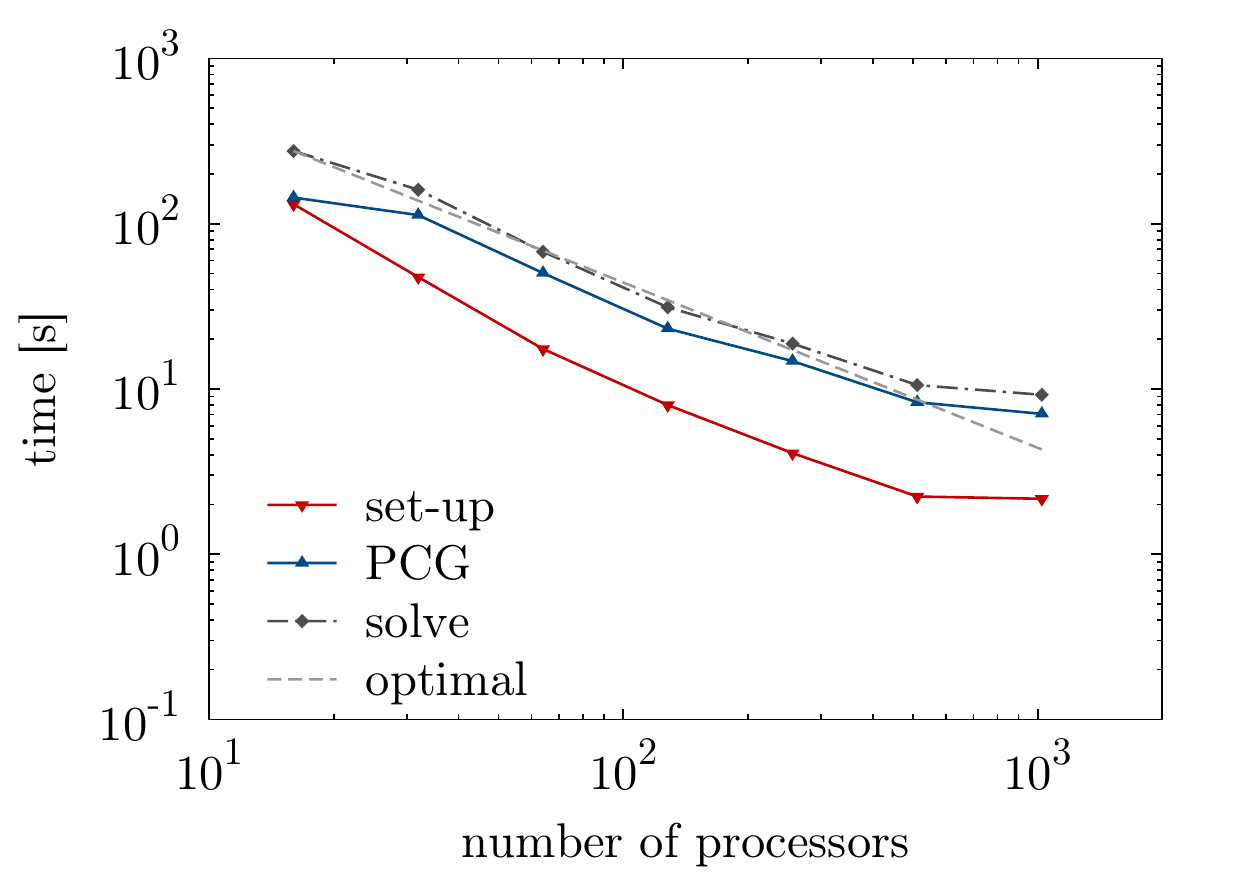} 
\includegraphics[width=0.48\textwidth]{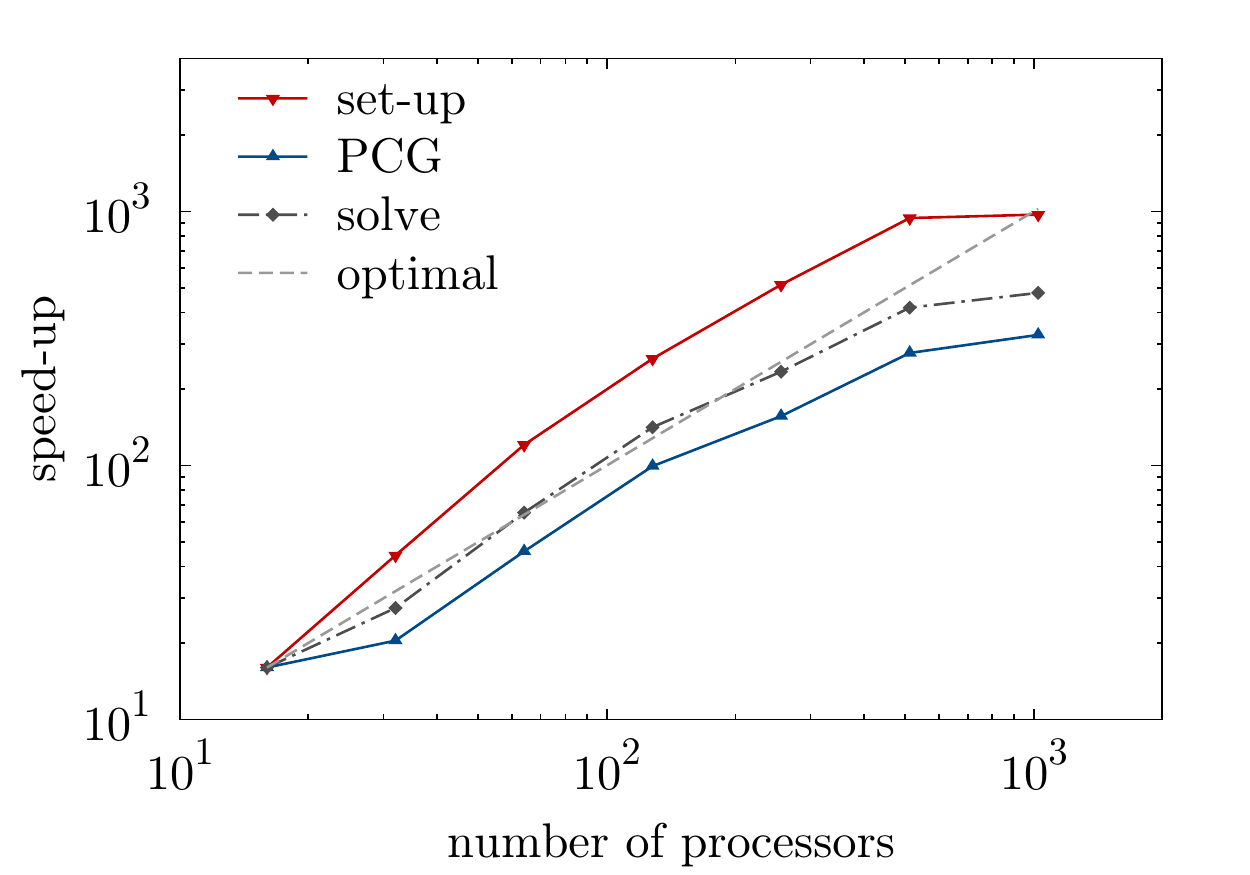}
\end{center}
\caption{\label{fig:timing_Melechov_large}
Strong scaling test for the problem of the \emph{Melechov locality}
containing 2D and 3D elements and 15M unknowns, computational time (left) and
speed-up (right) separately for set-up and PCG phases, and their sum (solve).}
\end{figure}

The second engineering model is the locality around the \emph{Bed\v{r}ichov tunnel}. 
The main purpose of this 2.1 km long tunnel 
near the city of Liberec in the north of the Czech Republic
to accommodate water pipes, which supply the city 
by drinking water from a reservoir in the mountains.
However, this locality is also a valuable site for experimental geological measurements performed inside the tunnel.

The model aims at describing the flow in the granite rock surrounding the
tunnel. The computational mesh consists of 1.1 million elements leading to 7.8
million unknowns. The mesh with the plot of resulting piezometric head is
presented in Fig.~\ref{fig:tunel_1}. The system of fractures and an example
division into 256 substructures are visualised in Fig.~\ref{fig:tunel_2}. The
hydraulic conductivity of the fractures is 10$^{-7}$ ms$^{-1}$, while
the conductivity of the outer material is 10$^{-10}$ ms$^{-1}$, the transition
coefficient $\sigma_{3}$ = 1 s$^{-1}$, and the effective thickness of fractures
$\delta_{2}$ = 1.1 m.

Although the mesh contains fewer finite elements than the one of the
\emph{Melechov locality} model, this problem is considerably more complicated. This is
caused mainly by the presence of relatively very small and irregularly shaped finite
elements in the vicinity of the tunnel and near the cross-sections of
fractures (see Fig.~\ref{fig:tunel_3}) generated by the mesh generator.

The results of a strong scaling test are summarised in Table~\ref{tab:tunnel_strong_scaling}. 
As before, the times are also plotted in Fig.~\ref{fig:timing_tunnel}.
Although the number of iterations is not independent of the number of
substructures, the growth is still small. Consequently, the computing times,
and especially the time for set-up, scale very well over a large range of
numbers of substructures. The observed super-optimal scaling may be related to
faster factorisation of the smaller local problems by the direct solver for
indefinite matrices.

Table~\ref{tab:tunnel_corner_effect} summarises an experiment 
performed to analyse the effect of using corners in the construction of
the coarse space in BDDC. As has been mentioned in Section~\ref{sec:bddc},
using Raviart-Thomas finite elements does not lead to `natural' corners as
cross-points shared by several substructures. On the other hand, the notion of
corners was generalised to any selected interface degree of freedom, at which
continuity of functions from the coarse space is required. Such generalisation
is important for the well-posedness of the local problems for unstructured
meshes e.g. in elasticity analysis \cite{Sistek-2012-FSC}. This is also the
default option for \textsl{BDDCML}, in which selection of corners is performed
at each face between two substructures. 
Adding corners improves the approximation properties of the coarse space at the cost of
increasing the size of the coarse problem.
Table~\ref{tab:tunnel_corner_effect} compares the convergence for variable
number of substructures without and with constraints at corners. 
Column `with corners' in Table~\ref{tab:tunnel_corner_effect} corresponds to results in Table~\ref{tab:tunnel_strong_scaling},
and it is repeated for comparison.
We can see
that while the effect of corners on convergence is small for smaller number of
substructures, the improvement of the coarse problem and the approximation of
BDDC becomes more significant for higher numbers of cores. Looking at times in
Table~\ref{tab:tunnel_corner_effect}, the additional time spent in the set-up
phase due to higher number of constraints when using corners is 
compensated by the lower number of PCG iterations, resulting in lower overall
times. Thus, using additionally selected corners appears beneficial for
complicated engineering problems like this one.

\begin{figure}[ptbh]
\begin{center}
\begin{overpic}[width=0.7\textwidth]{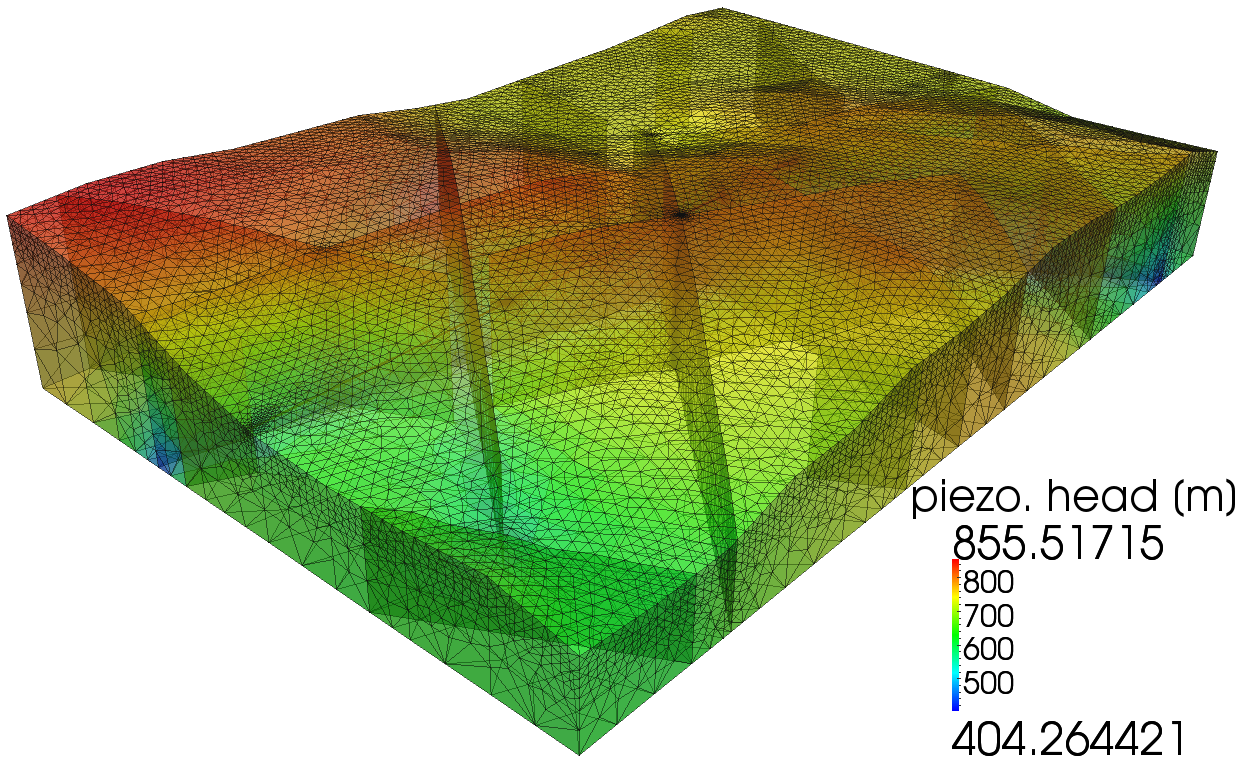}
\put(2,12){1754 m (width)}
\put(16,57){2655 m (length)}
\put(98,43){458 m (height)}
\end{overpic}
\end{center}
\caption{\label{fig:tunel_1}
The \emph{Bed\v{r}ichov tunnel} problem containing 2D and 3D
elements, mesh contains 1.1M elements and 7.8M unknowns. Plot of the
piezometric head. Data by courtesy of Dalibor Frydrych.}
\end{figure}

\begin{figure}[ptbh]
\begin{center}
\includegraphics[width=0.48\textwidth]{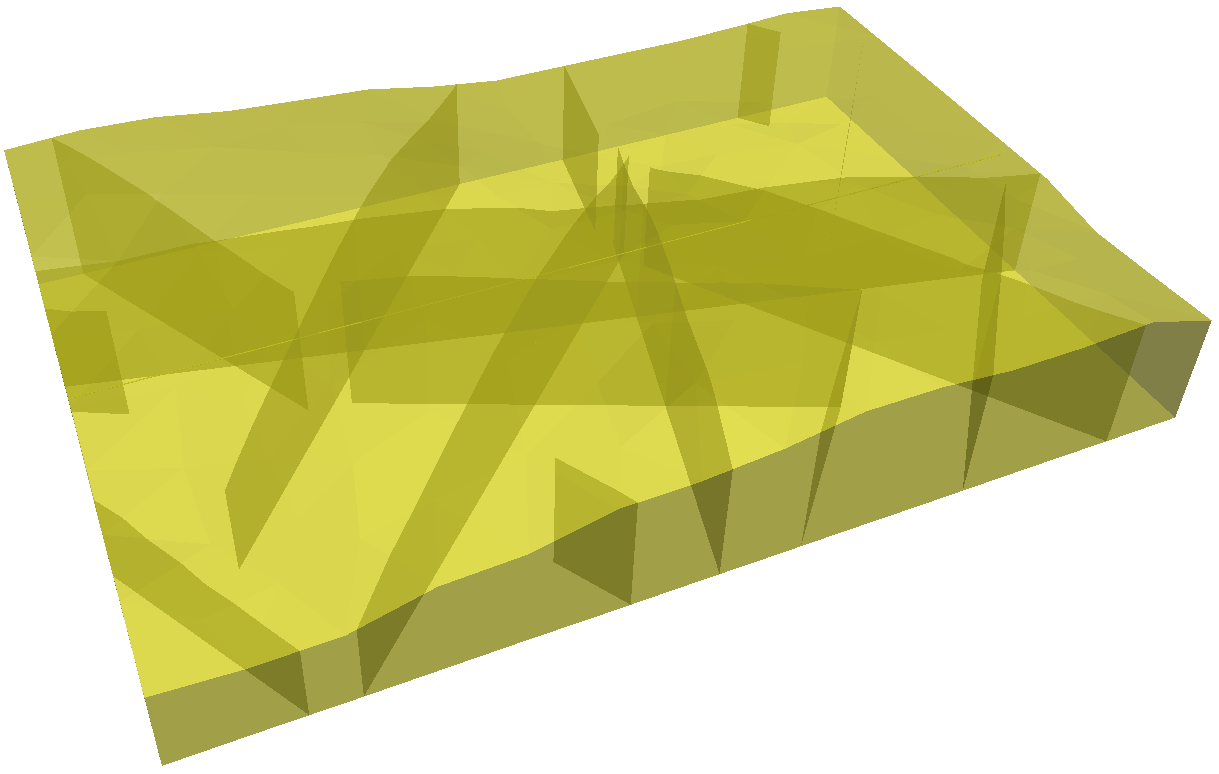} 
\includegraphics[width=0.48\textwidth]{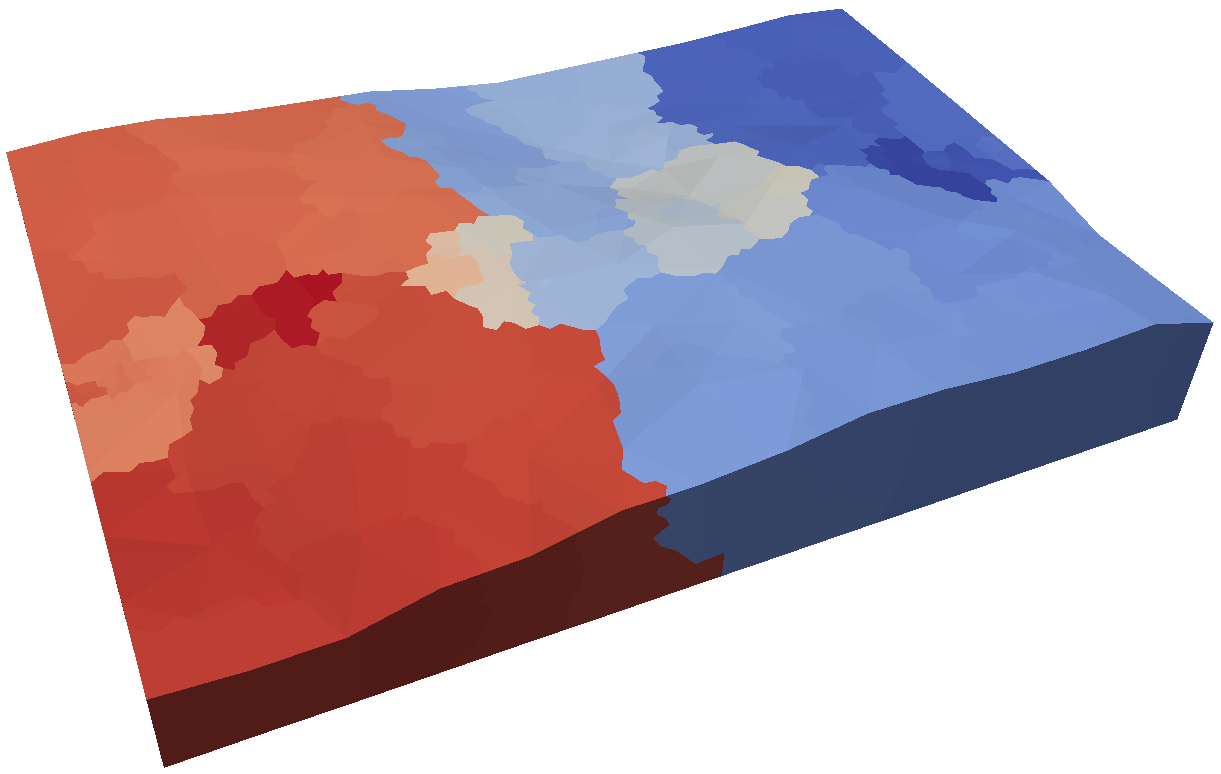}
\end{center}
\caption{\label{fig:tunel_2}
The \emph{Bed\v{r}ichov tunnel} problem; the system of fractures
(left) and an example division into 256 substructures (right).}
\end{figure}

\begin{figure}[ptbh]
\begin{center}
\includegraphics[width=0.455\textwidth]{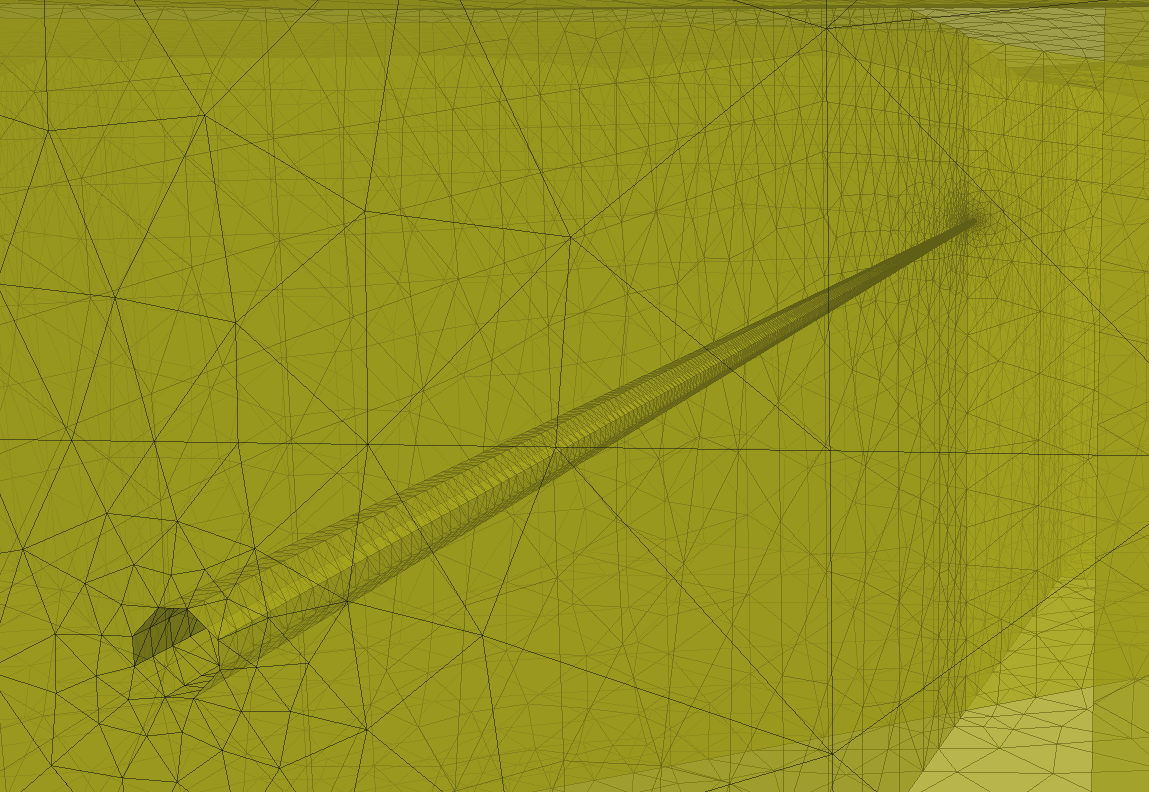} \ 
\includegraphics[width=0.48\textwidth]{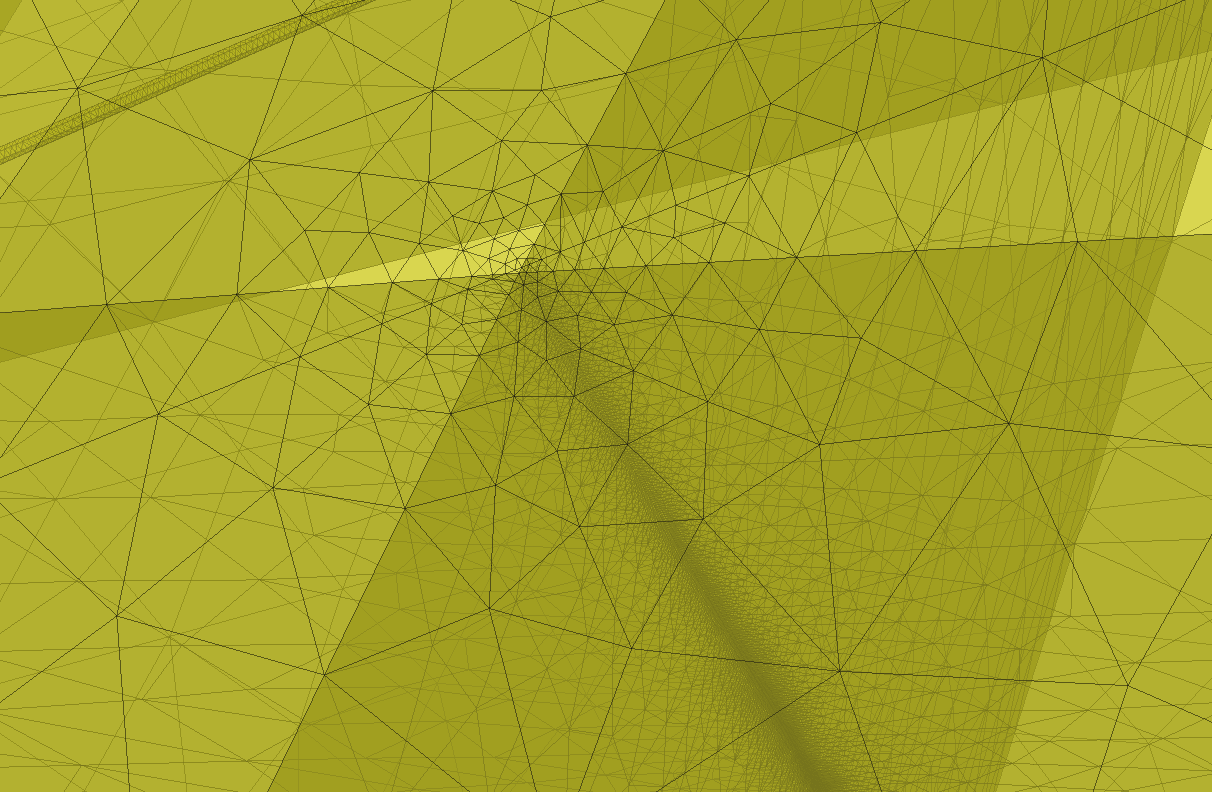}
\end{center}
\caption{\label{fig:tunel_3}
Example of difficulties in the mesh of the \emph{Bed\v{r}ichov
tunnel} problem; detail of the tunnel geometry with fine elements (left) and
enforced refinement at an intersection of fractures (right).}
\end{figure}

\begin{table}[pth]
\begin{center}
\begin{tabular}
[c]{|cc|ccc|cc|ccc|}\hline
\multirow{2}{*}{$N$} & \multirow{2}{*}{$n/N$} & \multirow{2}{*}{$n_{\Gamma}$} & \multirow{2}{*}{$n_f$} & \multirow{2}{*}{$n_c$} & 
\multirow{2}{*}{its.} & \multirow{2}{*}{cond.} & \multicolumn{3}{c|}{time (sec)} \\
&  &  &  &  &  &  & set-up & PCG & solve \\\hline 
32 & 245k & 20k & 106 & 322     & 112 & 1514.1 & 110.3 & 144.0 & 254.3\\
64 & 123k & 28k & 192 & 597     & 63  & 117.7  & 42.2  & 36.0  & 78.3\\
128 & 61k & 45k & 413 & 1293    & 75  & 194.4  & 13.4  & 16.8  & 30.3\\
256 & 31k & 72k & 902 & 2791    & 119 & 526.7  & 4.2   & 10.9  & 15.1\\
512 & 15k & 110k & 2009 & 6347  & 137 & 1143.4 & 1.8   & 7.1   & 9.0\\
1024 & 8k & 155k & 4575 & 14725 & 173 & 897.0  & 1.6   & 8.0   & 9.7\\\hline
\end{tabular}
\end{center}
\caption{\label{tab:tunnel_strong_scaling}
Strong scaling test for the problem of the \emph{Bed\v{r}ichov
tunnel} containing 2D and 3D elements, size of the global problem is $n$ = 7.8M unknowns.}
\end{table}

\begin{figure}[ptbh]
\begin{center}
\includegraphics[width=0.48\textwidth]{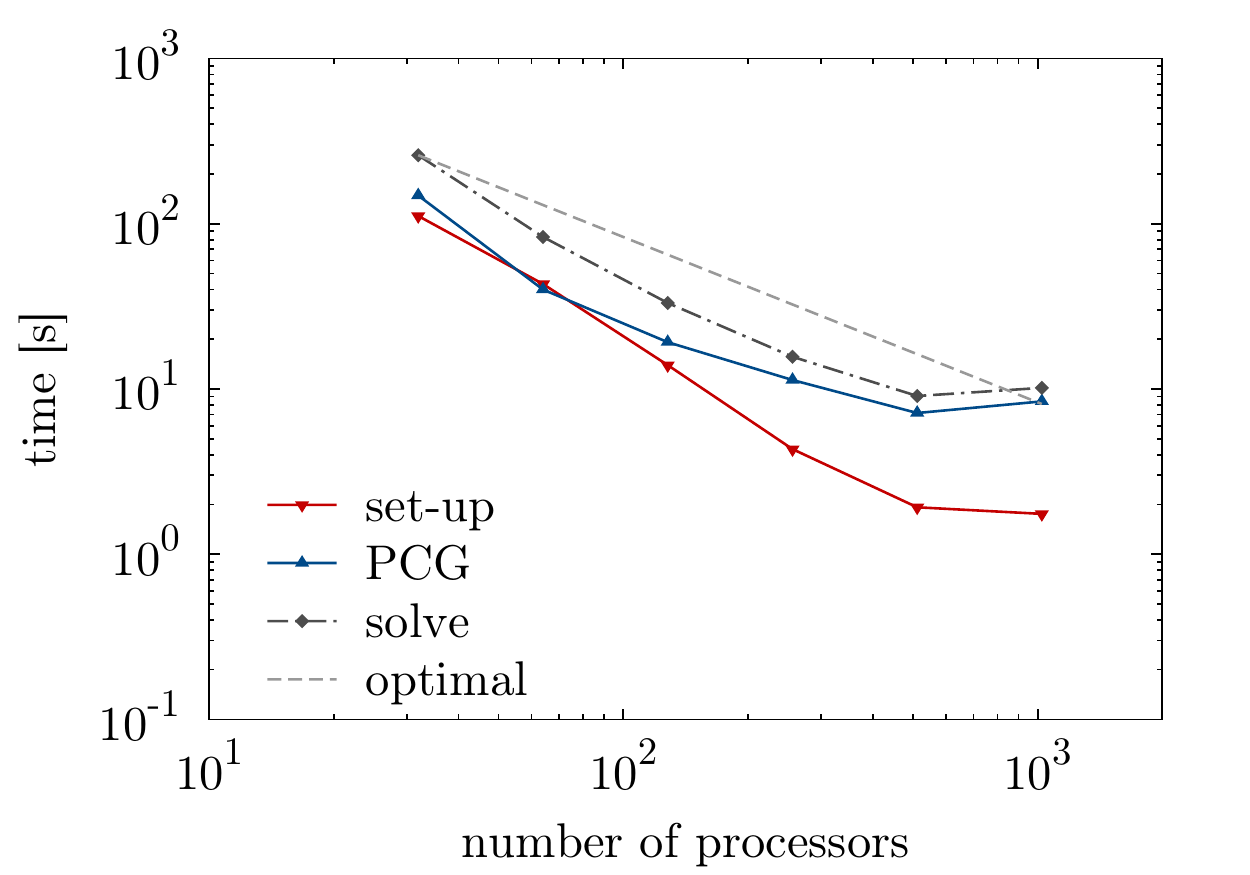} 
\includegraphics[width=0.48\textwidth]{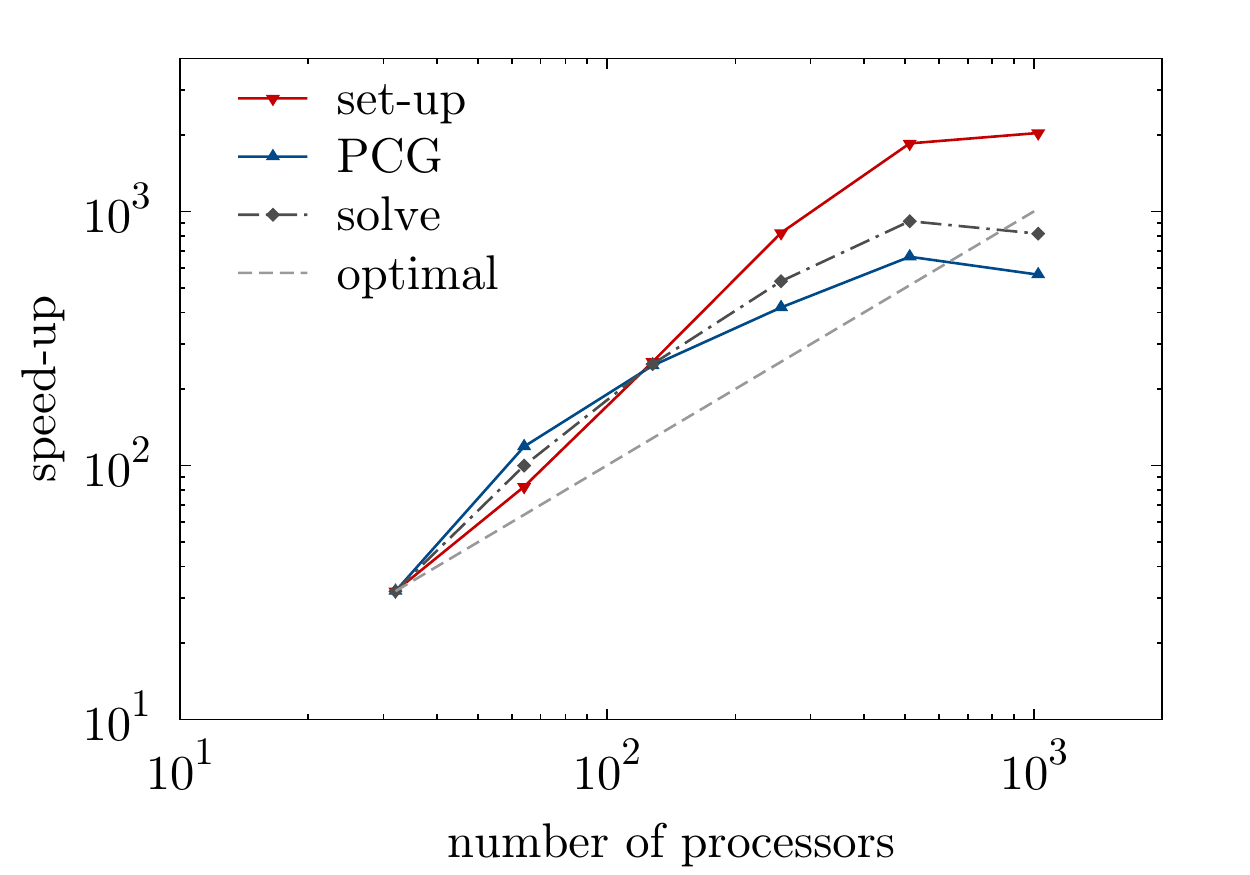}
\end{center}
\caption{\label{fig:timing_tunnel}
Strong scaling test for the problem of the \emph{Bed\v{r}ichov
tunnel} containing 2D and 3D elements and 7.8M unknowns, computational time
(left) and speed-up (right) separately for set-up and PCG phases, and their
sum (solve).}
\end{figure}

\begin{table}[pth]
\begin{center}
\begin{tabular}
[c]{|cc|ccccc|ccccc|}\hline
\multirow{3}{*}{$N$} & \multirow{3}{*}{$n/N$} & \multicolumn{5}{c|}{without corners} & \multicolumn{5}{c|}{with corners}\\
                     &                        & \multirow{2}{*}{its.}  & \multirow{2}{*}{cond.} & \multicolumn{3}{c|}{time (sec)} & 
                                                \multirow{2}{*}{its.}  & \multirow{2}{*}{cond.} & \multicolumn{3}{c|}{time (sec)}\\
   &      &     &        & set-up & PCG & solve &  &  & set-up & PCG & solve\\\hline
32 & 245k & 131 & 1789.6 & 107.5 & 175.0 & 282.5 & 112 & 1514.1 & 110.3 & 144.0 & 254.3\\
64 & 123k & 70 & 122.2 & 40.3 & 40.4 & 80.7 & 63 & 117.7 & 42.2 & 36.0 & 78.3\\
128 & 61k & 96 & 208.5 & 10.9 & 21.6 & 32.6 & 75 & 194.4 & 13.4 & 16.8 & 30.3\\
256 & 31k & 139 & 541.9 & 3.7 & 12.5 & 16.2 & 119 & 526.7 & 4.2 & 10.9 & 15.1\\
512 & 15k & 197 & 1418.9 & 1.4 & 10.0 & 11.4 & 137 & 1143.4 & 1.8 & 7.1 & 9.0\\
1024 & 8k & 312 & 3779.4 & 1.0 & 14.5 & 15.6 & 173 & 897.0 & 1.6 & 8.0 & 9.7\\\hline
\end{tabular}
\end{center}
\caption{\label{tab:tunnel_corner_effect}
Effect of using corners. Problem of the \emph{Bed\v{r}ichov tunnel} 
containing 2D and 3D elements, size of the global problem is $n$ =
7.8M unknowns. In column `without corners', no additional corners are selected
in BDDC. In column `with corners', additional corners are selected.}
\end{table}

In the final experiment, we compare the effect of different averaging
techniques on the convergence of BDDC. In Table~\ref{tab:tunnel_averaging},
results of the strong scaling test for arithmetic averaging~(\ref{eq:scaling_arithmetic}), 
the modified $\rho$-scaling~(\ref{eq:scaling_rho}), 
and the proposed scaling by diagonal stiffness~(\ref{eq:scaling_diagonal}) are summarised. 
The final column corresponds to
the results from Table~\ref{tab:tunnel_strong_scaling}, which are repeated here for comparison.

Table~\ref{tab:tunnel_averaging} suggests, that while the simple arithmetic
averaging does not lead to satisfactory convergence for this problem, the
modified $\rho$-scaling and the diagonal scaling mostly lead to similar
convergence. However, while the former provides slightly better convergence
for several cases, it also leads to irregularities for certain divisions, for
which the BDDC method with this averaging converges rather poorly. Therefore,
the proposed scaling (\ref{eq:scaling_diagonal}) can be recommended as the
most robust choice among the three tested options.

\begin{table}[ptbh]
\begin{center}
\begin{tabular}
[c]{|cc|ccc|cc|cc|cc|}\hline
\multirow{2}{*}{$N$} & \multirow{2}{*}{$n/N$} &
\multirow{2}{*}{$n_{\Gamma}$} & \multirow{2}{*}{$n_f$} &
\multirow{2}{*}{$n_c$} & \multicolumn{2}{c|}{arithmetic avg.} &
\multicolumn{2}{c|}{mod. $\rho$-scal.} & \multicolumn{2}{c|}{diagonal scal.}\\
&  &  &  &  & its. & cond. & its. & cond. & its. & cond.\\\hline
32 & 245k & 20k & 106 & 322 & 637 & 9811.7 & 110 & 1467.8 & 112 & 1514.1\\
64 & 123k & 28k & 192 & 597 & 618 & 10254.1 & 62 & 115.1 & 63 & 117.7\\
128 & 61k & 45k & 413 & 1293 & 2834 & 1.0e+11 & 206 & 401641.4 & 75 & 194.4\\
256 & 31k & 72k & 902 & 2791 & 799 & 11172.9 & 117 & 512.9 & 119 & 526.7\\
512 & 15k & 110k & 2009 & 6347 & 883 & 15449.6 & 136 & 1160.1 & 137 & 1143.4\\
1024 & 8k & 155k & 4575 & 14725 & n/a & 2.5e+10 & 504 & 99023.6 & 173 &
897.0\\\hline
\end{tabular}
\end{center}
\caption{\label{tab:tunnel_averaging}
Comparison of different averaging techniques for the
\emph{Bed\v{r}ichov tunnel} containing 2D and 3D elements, size of the global
problem is $n$ = 7.8M unknowns.}
\end{table}

\section{Conclusion}

\label{sec:conclusion}

A parallel solver for the mixed-hybrid finite element formulation based on Darcy's law has
been presented. The software combines an existing package \textsl{Flow123d}
developed for problems in geophysics with \textsl{BDDCML}, a parallel library
for solution of systems of linear algebraic equations by the BDDC method.

In geoengineering applications, the mathematical model is applied to
geometries with the presence of fractures. In the present approach, the flow
in these fractures is also modelled by Darcy's law, although the hydraulic
conductivity of the porous media is considered by orders of magnitude higher.
These fractures are modelled by finite elements of a lower dimension. In the
discretised model, 1D, 2D and 3D finite elements are coupled together through
Robin boundary conditions. These coupling terms lead to a modification of the
usual saddle-point matrix of the system, in which a new non-zero block appears
on the diagonal.

The BDDC method is employed for the solution of the resulting system of linear
algebraic equations. BDDC is based on iterative substructuring, in which the
problem is first reduced to the interface among substructures. The Schur
complement is not built explicitly. Instead, only multiplications by the
matrix are performed through solving a discrete Dirichlet problem on each
substructure. In the setting of the mixed-hybrid problem, the interface is
built only as a subset of the block of Lagrange multipliers, while remaining
unknowns belong to interiors of substructures. Although the original problem
is symmetric indefinite, the system reduced to the interface is symmetric
positive definite. This is also shown to hold for the case with fractures in
the present paper. Consequently, the PCG method is used for the solution of the
reduced problem. However, unlike the symmetric positive definite problems,
a direct solver for symmetric indefinite matrices needs to be used for
the factorisation and repeated solution of local problems on substructures.

One step of the BDDC method is used as the preconditioner for the PCG method
run on the interface problem. A modification of the diagonal stiffness scaling
has been introduced. It is motivated by difficult engineering problems,
for which it performs significantly better than other two applicable
choices---the arithmetic averaging and the modified $\rho$-scaling. Arithmetic
averages over faces between substructures are used as the basic constraints
defining the coarse space. 
In addition, corners are selected 
from unknowns at the interface 
using the face-based algorithm.
While corners are not
required by the theory, they are shown to improve both the convergence and the
computational times for complicated problems.

The performance of the resulting solver has been investigated on three
benchmark problems in 2D and 3D. Both weak and strong scaling tests have been
performed. On benchmark problems with single mesh dimension, 
the expected optimal convergence
independent of number of substructures has been achieved. Correspondingly, the
resulting parallel scalability has been nearly optimal for the weak scaling
tests up to 64 computer cores.

The strong scaling tests were presented for a benchmark problem of a unit cube
and for two engineering problems containing large variations in element sizes
and hydraulic conductivities, using up to 1024 computer cores and containing
up to 15 million degrees of freedom. The convergence for the unit cube problem
with all three possible dimensions of finite elements slightly deteriorated by
using more substructures, and this translated to sub-optimal parallel
performance. However, for the two engineering applications, in which only 3D
and 2D elements are combined, the BDDC method has also maintained good convergence
properties with the growing number of substructures, resulting in optimal, or
even super-optimal parallel scalability of the solver. It has been also shown 
that the proposed modification of the diagonal stiffness scaling plays an important role
in achieving such independence for the challenging engineering problems
presented in the paper.

\ack
The authors are grateful to Ji{\v r}ina Kr{\' a}lovcov{\' a} and Dalibor
Frydrych for providing the geoengineering models.

%\bibliographystyle{wileyj}
%\bibliography{bddc_hybrid}

\end{document}